\documentclass[12pt]{amsart}
\usepackage{tikz}
\usepackage{amsmath,amsthm,amssymb,xcolor,multicol}

\newtheorem{theorem}{Theorem}[section]
\newtheorem{lemma}[theorem]{Lemma}
\newtheorem{proposition}[theorem]{Proposition}
\newtheorem{corollary}[theorem]{Corollary}
\newtheorem{conjecture}[theorem]{Conjecture}

\newenvironment{definition}[1][Definition]{\begin{trivlist}
\item[\hskip \labelsep {\bfseries #1}]}{\end{trivlist}}
\newenvironment{example}[1][Example]{\begin{trivlist}
\item[\hskip \labelsep {\bfseries #1}]}{\end{trivlist}}

\newenvironment{problem}[1][Problem]{\begin{trivlist}
\item[\hskip \labelsep {\bfseries #1}]}{\end{trivlist}}

\newcommand{\len}{\operatorname{len}}
\newcommand{\ZZ}{\mathbb{Z}}

\newcommand{\ASn}{\tilde{S}}

\newcommand{\NDPF}{\operatorname{NDPF}}
\newcommand{\BNDPF}{\operatorname{BNDPF}}
\newcommand{\ANDPF}{\operatorname{NDPF^{(1)}}}

\newcommand{\leftexp}[2]{{\vphantom{#2}}^{#1}{#2}}

\newcommand{\N}{\mathbb{N}}

\newcommand{\JJ}{\mathcal{J}}

\usepackage{ifthen}
\newboolean{draft}
\setboolean{draft}{true}
\ifdraft
\newcommand{\TODO}[2][To do: ]{\textcolor{red}{\textbf{#1#2}}}
\else
\newcommand{\TODO}[2][]{}
\fi

\begin{document}
\title{Algebraic and Affine Pattern Avoidance}
\author{Tom Denton}
\date{}

\maketitle

{\abstract 
We investigate various connections between the $0$-Hecke monoid, Catalan monoid, and pattern avoidance in permutations, providing new tools for approaching pattern avoidance in an algebraic framework.  In particular, we characterize containment of a class of `long' patterns as equivalent to the existence of a corresponding factorization.  We then generalize some of our constructions to the affine setting.
}

\tableofcontents

\section{Introduction}

Pattern avoidance is a rich and interesting subject which has received much attention since Knuth first connected the notion of $[231]$-avoidance with stack sortability~\cite{knuth.TAOCP1}.  Pattern avoidance has also appeared in the study of smoothness of Schubert varieties~\cite{billeyLakshmibai.2000, Billey98patternavoidance}, the Temperley-Lieb algebra and the computation of Kazhdahn-Lusztig polynomials~\cite{fan.1996, fanGreen.1999}.  There is also an extensive literature on enumeration of permutations avoiding a given pattern; for an introduction, see~\cite{bona.permutations}.  
Pattern containment (the complementary problem to pattern avoidance) was previously known to be related to the strong Bruhat order; in particular, Tenner showed that a principal order ideal of a permutation is Boolean if and only if the permutation avoids the patterns $[321]$ and $[3412]$~\cite{tenner.2007}.  

While many have studied pattern avoidance for particular patterns, there are relatively few results on pattern avoidance as a general phenomenon.  Additionally, while there has been a great deal of combinatorial research on pattern avoidance, there have been few algebraic characterizations.  
In this paper, we first introduce an equivalence between pattern containment and a factorization problem for certain permutation patterns.  We then use these results directly in analysing the fibers certain quotients of the $0$-Hecke monoid.  Finally, we consider the question of pattern avoidance in the affine permutation group.

We begin by introducing the notion of a width system, which, in some cases, allows the factorization of a permutation $x$ containing a pattern $\sigma$ as $x=y\sigma'z$, where $\sigma'$ is a `shift' of $\sigma$, $y$ and $z$ satisfy certain compatibility requirements, and the $\len(x)=\len(y)+\len(\sigma)+\len(z)$.  This factorization generalizes an important result of Billey, Jockusch, and Stanley~\cite{BilleyJockuschStanley.1993}, which states that any permutation $x$ containing a $[321]$-pattern contains a braid; that is, some reduced word for $x$ in the simple transpositions contains a contiguous subword $s_i s_{i+1} s_i$.  (This subword, in our context, plays the role of the $\sigma'$.)  Equivalently, a permutation that is $[321]$-avoiding is fully commutative, meaning that every reduced word may be obtained by commutation relations.  These permutations have been extensively studied, with major contributions by Fan and Green~\cite{fan.1996, fanGreen.1999} and Stembridge~\cite{stembridge.1996}, who associated a certain poset to each fully commutative element, where linear extensions of the poset are in bijection with reduced words for the permutation. 

Width systems allow us to extend this notion of subword containment considerably, and give an algebraic condition for pattern containment for certain patterns.  The width system is simply a measure of various widths of a pattern occurrence within a permutation (called an `instance').  For certain width systems, an instance of minimal width implies a factorization of the form discussed above.  These width systems tend to exist for relatively long permutations.  The main results are contained in Propositions~\ref{prop:s2widthSystems}, \ref{prop:s3widthSystems}, \ref{prop:widthSystemExtend}, \ref{prop:widthSystemExtend2}, and Corollary~\ref{cor:bountifulPerms}.  

We then apply these results directly, and study pattern avoidance of certain patterns (most interestingly  $[321]$-avoidance) in the context of quotients of the $0$-Hecke monoid.
The non-decreasing parking functions $\NDPF_N$ may be realized as a quotient of the $0$-Hecke monoid for the symmetric group $S_N$, and coincide with the set of order-preserving regressive functions on a poset when the poset is a chain.  These functions are enumerated by the Catalan numbers; if one represents $f \in \NDPF_N$ as a step function, its graph will be a (rotated) Dyck path.  These functions form a $\JJ$-trivial monoid under composition, and may be realized as a quotient of the $0$-Hecke monoid; the monoid $\NDPF_n$ coincides with the Catalan monoid.  We show that the fibers of this quotient each contain a unique $[321]$-avoiding permutation of minimal length and a $[231]$-avoiding permutation of maximal length (Theorem ~\ref{thm:ndpfFibers231}).  We then show that a slightly modified quotient has fibers containing a unique $[321]$-avoiding permutation of minimal length, and a $[312]$-avoiding permutation of maximal length (Theorem~\ref{thm:ndpfFibers312}).

This provides a bijection between $[312]$ and $[321]$-avoiding permutations.  The bijection is equivalent to the bijection of Simion and Schmidt between $[132]$-avoiding permutations and $[123]$-avoiding permutations~\cite{simion.schmidt.1985}, but here we have given an algebraic interpretation of the bijection.  (The patterns $[312]$ and $[123]$ are the respective ``complements'' of the patterns $[312]$ and $[321]$.)

We then combine these results to obtain a bijection between $[4321]$-avoiding permutations and elements of a submonoid of $\NDPF_{2N}$ (Theorem~\ref{thm:ndpfFibers4321}), which we consider as a parabolic submonoid of a type $B$ generalization of non-decreasing parking functions, which coincide with the double Catalan monoid~\cite{mazorchukSteinberg2011}.

We then expand our discussion to the affine symmetric group and affine $0$-Hecke monoid.  The affine symmetric group was introduced originally by Lusztig~\cite{lusztig.1983}, and questions concerning pattern avoidance in the affine symmetric group have recently been studied by Lam~\cite{Lam06affinestanley}, Green~\cite{Green.2002}, Billey and Crites~\cite{billeyCrites.2011}.  Lam and Green separately showed that an affine permutation contains a $[321]$-pattern if and only if it contains a braid, in the same sense as in the finite case.  

We introduce a definition for affine non-decreasing parking functions $\ANDPF_N$, and demonstrate that this monoid of functions may be obtained as a quotient of the affine symmetric group.  We obtain a combinatorial map from affine permutations to $\ANDPF_N$ and demonstrate that this map coincides with the definition of $\ANDPF_N$ by generators and relations as a quotient of $\ASn_N$.  Finally, we prove that each fiber of this quotient contains a unique $[321]$-avoiding element of minimal length (Theorem~\ref{thm:affNdpfFibers321}).  

\subsection{Overview.}

In Section~\ref{sec:widthSystems} we introduce \textbf{width systems} on permutation patterns as a potential system for understanding pattern containment algebraically.  The main results of this section describe a class of permutation patterns $\sigma$ such that any permutation $x$ containing $\sigma$ factors as $x=y\sigma' z$, with $\len(x)=\len(y)+\len(\sigma)+\len(z)$.  Here $\sigma'$ is a ``shift'' of $\sigma$, and some significant restrictions on $y$ and $z$ are established.  
The main results are contained in Propositions~\ref{prop:s2widthSystems}, \ref{prop:s3widthSystems}, \ref{prop:widthSystemExtend},\ref{prop:widthSystemExtend2}, and Corollary~\ref{cor:bountifulPerms}.  

We apply these ideas directly in Section~\ref{sec:ndpfPattAvoid} while analyzing the fiber of a certain quotient of the $0$-Hecke monoid of the symmetric group.  In Theorem~\ref{thm:ndpfFibers231}, we show that each fiber of the quotient contains a unique $[321]$-avoiding permutation and a unique $[231]$-avoiding permutation.  We then apply an involution and study a slightly different quotient in which fibers contain a unique $[321]$-avoiding permutation and a unique $[312]$-avoiding permutation (Theorem~\ref{thm:ndpfFibers312}).  In Section~\ref{sec:bndpfPattAvoid}, we consider a different monoid-morphism of the $0$-Hecke monoid for which each fiber contains a unique $[4321]$-avoiding permutation (Theorem~\ref{thm:ndpfFibers4321}).

We then define the Affine Nondecreasing Parking Functions in Section~\ref{sec:affNdpfPattAvoid}, and establish these as a quotient of the $0$-Hecke monoid of the affine symmetric group.  We prove the existence of a unique $[321]$-avoiding affine permutation in each fiber of this quotient (Theorem~\ref{thm:affNdpfFibers321}). 

\subsection{Acknowledgements.}

This paper originally appeared as a chapter in the author's PhD thesis, awarded by the University of California, Davis.  As such thanks are due to my co-advisors, Prof. Anne Schilling and Nicolas M. Thi\'ery, as well as my committee members, who provided useful feedback during the writing process.  Thanks are also due to the incredible math department at Davis, which provided a fertile ground for study for five years.  As I prepare this paper, I am a postdoctoral researcher at York University.  Additional support (and copious amounts of coffee) is provided by the Fields Institute.

\section{Background and Notation}
\label{sec:bgnot}

\subsection{Pattern Avoidance.}

Pattern avoidance phenomena have been studied extensively, originally by Knuth in his 1973 classic, The Art of Computer Programming~\cite{knuth.TAOCP1}.  A thorough introduction to the subject may be found in the book ``Combinatorics of Permutations'' by Bona~\cite{bona.permutations}.  A \textbf{pattern} $\sigma$ is a permutation in $S_k$ for some $k$; given a permutation $x \in S_N$, we say that $x$ \textbf{contains the pattern $\sigma$} if, in the one-line notation for $x=[x_1, \ldots, x_N]$, there exists a subsequence $[x_{i_1}, \ldots, x_{i_k}]$ whose elements are in the same relative order as the elements in $p$.  If $x$ does not contain $\sigma$, then we say that 
$x$ \textbf{avoids $\sigma$}, or that $x$ is \textbf{$\sigma$-avoiding.}  (Note that if $k>N$, $x$ must avoid $\sigma$.)

For example, the pattern $[1, 2]$ appears in any $x$ such that there exists a $x_i < x_j$ for some $i<j$.  The only $[12]$-avoiding permutation in $S_N$, then, is the long element, which is strictly decreasing in one-line notation.  As a larger example, the permutation $[\mathbf{3}, \mathbf{4}, 5, \mathbf{2}, 1, 6]$ contains the pattern $[231]$ at the bold positions.  In fact, this permutation contains six distinct instances of the pattern $[231]$.

An interesting and natural question is, given a pattern $\sigma$, how many permutations in $S_N$ avoid $\sigma$?  It has been known since Knuth's original work that for any pattern in $S_3$, there are Catalan-many permutations in $S_N$ avoiding $\sigma$~\cite{knuth.TAOCP1}.

The $[321]$-avoiding permutations are of particular importance.  It was shown in~\cite{BilleyJockuschStanley.1993} that a permutation $x\in S_N$ is $[321]$-avoiding if and only if $x$ is `braid free.'  In particular, this means that there is no reduced word for $x$ containing the consecutive subsequence of $s_i s_{i+1} s_i$ (or $s_{i+1} s_i s_{i+1}$, equivalently), where the $s_i$ are the simple transpositions generating $S_N$.  Such permutations are called \textbf{fully commutative}. 

Lam~\cite{Lam06affinestanley} and Green~\cite{Green.2002} separately showed that this result extends to the  affine symmetric group.  The affine symmetric group (see Definition~\ref{def:affSn}) is a subset of the permutations of $\ZZ$, satisfying some periodicity conditions.  Pattern avoidance for the affine symmetric group works exactly as in a finite symmetric group.  
The one-line notation for $x$ is the doubly infinite sequence $x=[\ldots, x_{-1}, x_0, x_1, \ldots, x_N, x_{N+1}, \ldots]$.
Then $x$ contains a pattern $\sigma$ if any subsequence of $x$ in one-line notation has the same relative order as $\sigma$.  \textbf{Fully commutative elements} of the affine symmetric group are those which have no reduced word containing the consecutive subsequence $s_i s_{i+1} s_i$, where the indices are considered modulo $N$.
Green showed that the fully commutative elements of the affine symmetric group coincide with the $[321]$-avoiding affine permutations.

Fan and Green~\cite{fan.1996, fanGreen.1999} previously studied the quotient of the full Hecke algebra $H_q(W)$ for $W$ simply-laced, by the ideal $I$ generated by $T_{sts} + T_{st} + T_{ts} + T_s + T_t + 1$ for $s$ and $t$ generators of $W$ satisfying a braid relation $sts=tst$.  This quotient $H/I$ yields the \textbf{Temperley-Lieb Algebra}.  Fan showed that this quotient has a basis indexed by fully commutative elements of $W$, and in further work with Richard Green derived information relating this quotient to the Kazhdan-Lusztig basis for $H_q(W)$.

A further application of pattern avoidance occurs in the study of rational smoothness of Schubert varieties; an introduction to this topic may be found in~\cite{billeyLakshmibai.2000}.  The Schubert varieties $X_w$ in Type $A$ are indexed by permutations; a result of Billey~\cite{Billey98patternavoidance} shows that $X_w$ is smooth if and only if $w$ is simultaneously $[3412]$- and $[4231]$-avoiding.  More recently, Billey and Crites have extended this result to affine Schubert varieties (for affine Type A)~\cite{billeyCrites.2011}, showing that an affine Schubert variety $X_w$ is rationally smooth if and only if $w$ is simultaneously $[3412]$- and $[4231]$-avoiding or is a special kind of affine permutation, called a ``twisted spiral.''  

\subsection{$0$-Hecke monoids}
\label{ssec:zeroHeckeDefinition}

Let $W$ be a finite Coxeter group with index set $I=\{1, \ldots, N-1\}$. It has a presentation
\begin{equation}
  W = \langle\, s_i \; \text{for} \; i\in I\ \text{such that }\ (s_is_j)^{m(s_i,s_j)},\ \forall i,j\in I\,\rangle\,,
\end{equation}
where $I$ is a finite set, $m(s_i,s_j) \in \{1,2,\dots,\infty\}$, and $m(s_i,s_i)=1$.
The elements $s_i$ with $i\in I$ are called \emph{simple reflections}, and the
relations can be rewritten as:
\begin{equation}
  \begin{alignedat}{2}
    s_i^2 &=1 &\quad& \text{ for all $i\in I$}\,,\\
    \underbrace{s_is_js_is_js_i \cdots}_{m(s_i,s_j)} &=
    \underbrace{s_js_is_js_is_j\cdots}_{m(s_i,s_j)} && \text{ for all $i,j\in I$}\, ,
  \end{alignedat}
\end{equation}
where $1$ denotes the identity in $W$. An expression $w=s_{i_1}\cdots s_{i_\ell}$ for $w\in W$ 
is called \emph{reduced} if it is of minimal length $\ell$.
See~\cite{Bjorner_Brenti.2005, Humphreys.1990} for further details on Coxeter groups.

The Coxeter group of type $A_{N-1}$ is the symmetric group $S_N$ with generators
$\{s_1,\dots,s_{N-1}\}$ and relations:
\begin{equation}
  \begin{alignedat}{2}
    s_i^2           & = 1                &     & \text{ for } 1\leq i\leq n-1\,,\\
    s_i s_j         & = s_j s_i            &     & \text{ for } |i-j|\geq2\,, \\
    s_i s_{i+1} s_i & = s_{i+1} s_i s_{i+1} &\quad& \text{ for } 1\leq i\leq n-2\,;
  \end{alignedat}
\end{equation}
the last two relations are called the \emph{braid relations}.

\begin{definition}[\textbf{$0$-Hecke monoid}]
The $0$-Hecke monoid $H_0(W) = \langle \pi_i \mid i \in I \rangle$ of a Coxeter group $W$ 
is generated by the \emph{simple projections} $\pi_i$ with relations
\begin{equation}
  \begin{alignedat}{2}
    \pi_i^2 &=\pi_i &\quad& \text{ for all $i\in I$,}\\
    \underbrace{\pi_i\pi_j\pi_i\pi_j\cdots}_{m(s_i,s_{j})} &=
    \underbrace{\pi_j\pi_i\pi_j\pi_i\cdots}_{m(s_i,s_{j})} && \text{ for all $i,j\in I$}\ .
  \end{alignedat}
\end{equation}
Thanks to these relations, the elements of $H_0(W)$ are canonically
indexed by the elements of $W$ by setting $\pi_w :=
\pi_{i_1}\cdots\pi_{i_k}$ for any reduced word $i_1 \dots i_k$ of $w$.
\end{definition}

\subsection{Non-decreasing Parking Functions}
\label{ssec:ndpf}

We consider a collection of functions which form a monoid under composition.
  Notice that we use the right action in this paper, so
that for $x\in P$ and a function $f:P\to P$ we write $x.f$ for the value of
$x$ under $f$. 

\begin{definition}[\textbf{Monoid of Non-Decreasing Parking Functions}]
  Let $P=\{1, \ldots, N+1\}$ be a poset. The set $\NDPF_{N+1}$ of functions $f: P \to P$ which are
  \begin{itemize}
  \item \emph{order preserving}, that is, for all $x,y\in P,\ x\leq_P y$ implies
    $x.f\leq_P y.f$
  \item \emph{regressive}, that is, for all $x\in P$ one has $x.f \leq_P x$
  \end{itemize}
  is a monoid under composition.
\end{definition}
\begin{proof}
  It is trivial that the identity function is order preserving and regressive and that
  the composition of two order preserving and regressive functions is as well.
\end{proof}

According to~\cite[14.5.3]{Ganyushkin_Mazorchuk.2009}, not much is
known about these monoids.

When $P$ is a chain on $N$ elements, we obtain the monoid $\NDPF_N$ of
nondecreasing parking functions on the set $\{1, \ldots, N\}$ (see
e.g.~\cite{Solomon.1996}; it also is described under the notation
$\mathcal C_N$ in e.g.~\cite[Chapter~XI.4]{Pin.2009} and, together
with many variants, in~\cite[Chapter~14]{Ganyushkin_Mazorchuk.2009}).
The unique minimal set of generators for $\NDPF_N$ is given by the
family of idempotents $(\pi_i)_{i\in\{1,\dots,N-1\}}$, where each
$\pi_i$ is defined by $(i+1).\pi_i:=i$ and $j.\pi_i:=j$ otherwise. The
relations between those generators are given by:
\begin{gather*}
  \pi_i\pi_j = \pi_j\pi_i \quad \text{ for all $|i-j|>1$}\,,\\
  \pi_i\pi_{i-1}=\pi_i\pi_{i-1}\pi_i=\pi_{i-1}\pi_i\pi_{i-1}\,.
\end{gather*}
It follows that $\NDPF_N$ is the natural quotient of $H_0(S_N)$ by
the relation $\pi_i\pi_{i+1}\pi_i = \pi_{i+1}\pi_i$, via the quotient
map $\pi_i\mapsto
\pi_i$~\cite{Hivert.Thiery.HeckeSg.2006,Hivert.Thiery.HeckeGroup.2007,
  Ganyushkin_Mazorchuk.2010}. Similarly, it is a natural quotient of
Kiselman's
monoid~\cite{Ganyushkin_Mazorchuk.2010,Kudryavtseva_Mazorchuk.2009}.  
In~\cite{dhst.2011}, this monoid was studied as an instance of the larger class of order-preserving regressive functions on monoids, and a set of explicit orthogonal idempotents in the algebra was described.

\section{Width Systems, Pattern Containment, and Factorizations.}
\label{sec:widthSystems}

In this section we introduce width systems on permutation patterns, which sometimes provide useful factorizations of a permutation containing a given pattern.  The results established here will be directly applied in Sections~\ref{sec:ndpfPattAvoid} and~\ref{sec:bndpfPattAvoid}. 

\begin{definition}
Let $x$ be a permutation and $\sigma \in S_k$ a pattern.  We say that $x$ \textbf{factorizes over $\sigma$} if there exist permutations $y$, $z$, and $\sigma'$ such that:
\begin{enumerate}
\item $x = y \sigma' z$,
\item $\sigma'$ has a reduced word matching a reduced word for $\sigma$ with indices shifted by some $j$,
\item The permutation $y$ satisfies $y^{-1}(j)<\cdots < y^{-1}(j+k)$, 
\item The permutation $z$ satisfies $z(j)<\cdots < z(j+k)$,
\item $\len(x) = \len(y) + \len(\sigma') + \len(z)$.
\end{enumerate}
\end{definition}

Set $W=S_N$ and $J\subset I$, with $I$ the generating set of $W$.  An element $x\in W$ has a \textbf{right descent} $i$ if $\len(x s_i)<\len(x)$, and has a \textbf{left descent} $i$ if $\len(s_i x)<\len(x)$.  Equivalently, $x$ has a right (resp., left) descent at $i$ if and only if some reduced word for $x$ ends (resp., begins) with $i$.  Let $W^J$ be the set of elements in $W$ with no right descents in $J$.  Similarly, $\leftexp{J}{W}$ consists of those elements with no left descents in $J$.  Finally, $W_J$ is the \textbf{parabolic subgroup} of $W$ generated by $\{ s_i \mid i \in J\}$.

Recall that a \textbf{reduced word} or \textbf{reduced expression} for a permutation $x$ is a minimal-length expression for $x$ as a product of the simple transpositions $s_i$.  Throughout this chapter, we will use double parentheses enclosing a sequence of indices to denote words.  For example, $((1,3,2))$ corresponds to the element  $s_1s_2s_3$ in $S_4$.  Note that same expression can also indicate an element of $H_0(S_4)$, with $((1,3,2))$ corresponding to the element $\pi_1\pi_2\pi_3$.  Context should make usage clear.

\begin{definition}
Let $\sigma$ be a permutation pattern in $S_k$, with reduced word $((i_1, \ldots, i_m))$.  Let $J=\{j, j+1, \ldots, j+l\}$ for some $l\geq k-1$ and $\sigma' \in W_J$ with reduced word $((i_1+j, \ldots, i_m+j))$.  Then we call $\sigma'$ a \textbf{$J$-shift} or \textbf{shift} of $\sigma$.
\end{definition}

\begin{proposition}
A permutation $x\in S_N$ factorizes over $\sigma$ if and only if $x$ admits a factorization $x=y\sigma' z$ with $y \in W^J, \sigma' \in W_J$, and $z\in \leftexp{J}{W}$, and $\len(x) = \len(y) + \len(\sigma') + \len(z)$.
\end{proposition}
\begin{proof}
This is simply a restatement of the definition of factorization over $\sigma$.  In particular, $y\in W^J$ and $z\in \leftexp{J}{W}$.
\end{proof}
This condition is illustrated diagrammatically in Figure~\ref{fig.patternContainment} using a string-diagram for the permutation $x$ factorized as $y \sigma' z$.  In the string diagram of a permutation $x$, a vertical string connects each $j$ to $x(j)$, with strings arranged so as to have as few crossings as possible.
Composition of permutations is accomplished by vertical concatenation of string diagrams.    In the diagram, $x$ is the vertical concatenation (and product of) of $y$, $\sigma'$ and $z$.

The permutation $y^{-1}$ preserves the order of $\{j, j+1, \ldots, j+k\}$, and thus the strings leading into the elements $\{j, j+1, \ldots, j+k\}$ do not cross.  Likewise, $z$ preserves the order of $\{j, j+1, \ldots, j+k\}$, and thus the strings leading out of $\{j, j+1, \ldots, j+k\}$ in $z$ do not cross.  In between, $\sigma'$ rearranges $\{j, j+1, \ldots, j+k\}$ according to the pattern $\sigma$.

\begin{figure}
  \begin{center}
  \includegraphics[scale=1]{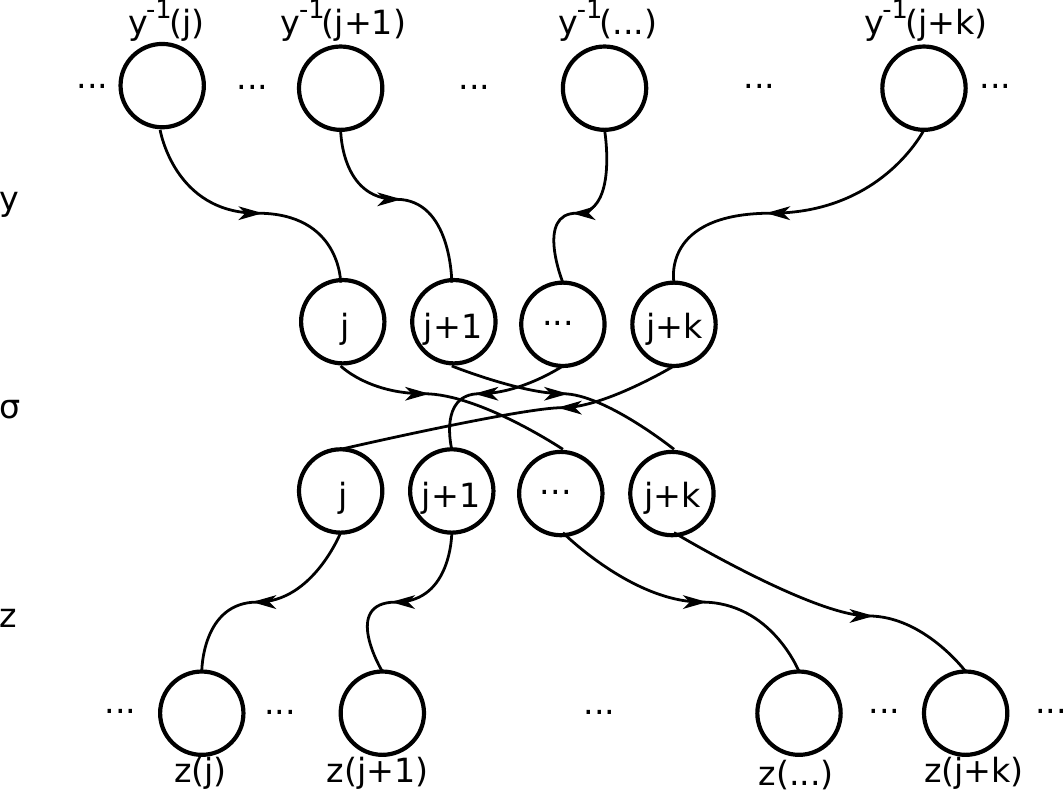}
  \end{center}
  \caption{Diagrammatic representation of a permutation $x$ factorizing over a pattern $\sigma$ as $x=y\sigma z$ by composition of string diagrams. }
  \label{fig.patternContainment}
\end{figure}

By the above discussion, it is clear that if $x$ admits a factorization $y\sigma' z$ with $y \in W^J, \sigma' \in W_J$, and $z\in \leftexp{J}{W}$ then $x$ contains $\sigma$.  The question, then, is when this condition is sharp.  This question is interesting because it provides an algebraic description of pattern containment.  For example, a permutation $x$ which contains a $[321]$-pattern is guaranteed to have a reduced expression which contains a braid.  Braid containment can be re-stated as a factorization over $[321]$.  When the factorization question is sharp, (ie, $x$ contains $\sigma$ if and only if $x$ factorizes over $\sigma$) one obtains an algebraic description of $\sigma$-containment.  The class of patterns with this property is rather larger than just $[321]$, as we will see in Propositions~\ref{prop:s2widthSystems}, \ref{prop:s3widthSystems}, and~\ref{prop:widthSystemExtend}.

\begin{problem}
For which patterns $\sigma$ does $x$ contain $\sigma$ if and only if $x\in W^J \sigma' \leftexp{J}{W}$, where $\sigma'$ is a $J$-shift of $\sigma$ for some $J$?
\end{problem}

As a tool for attacking this problem, we introduce the notion of a width system for a pattern.

\begin{definition}
Suppose $x$ contains $\sigma$ at positions $(i_1, \ldots, i_k)$; the tuple $P=(P_1, \ldots, P_k)$ is called an \textbf{instance} of the pattern $\sigma$, and we denote the set of all instances of $\sigma$ in $x$ by $P_x$.
\end{definition}

\begin{definition}
A \textbf{width} on an instance $P$ of $\sigma$ is a difference $P_j-P_i$ with $j>i$.  
A \textbf{width system} $w$ for a permutation pattern $\sigma \in S_k$ is a function assigning a tuple of widths to each instance of $\sigma$ in $x$.  An instance $P$ of a pattern in $x$ is \textbf{minimal} (with respect to $\sigma$ and $w$) if $w(P)$ is lexicographically minimal amongst all instances of $\sigma$ in $x$.  Finally, an instance $P=(P_1, \ldots, P_k)$ is \textbf{locally minimal} if $P$ is the minimal instance of $\sigma$ in the partial permutation $[x_{P_1}, x_{P_1+1}, \ldots, x_{P_k-1}, x_{P_k}]$.
\end{definition}

\begin{example}
Consider the pattern $[231]$ and let $P=(p, q, r)$ be an arbitrary instance of $\sigma$ in a permutation $x$.  We choose to consider the width system $w(P)=(r-p, q-p)$.  (Other width systems include $u(P)=(r-q, q-p)$ and $v(P)=(r-q)$, for example.)

The permutation $x=[3, 4, 5, 2, 1, 6]$ contains six $[231]$ patterns.  The following table records each $[231]$-instance $P$ and the width of the instance $w(P)$:
\begin{equation*}
\begin{array}[b]{|c c c|}
 \hline           &         P     &  w(P)  \\ \hline
 \left[\mathbf{3}, \mathbf{4}, 5, \mathbf{2}, 1, 6\right]       & (1, 2, 4) & (3, 1) \\
 \left[\mathbf{3}, \mathbf{4}, 5, 2, \mathbf{1}, 6\right]       & (1, 2, 5) & (4, 1) \\
 \left[\mathbf{3}, 4, \mathbf{5}, \mathbf{2}, 1, 6\right]       & (1, 3, 4) & (3, 2) \\
 \left[\mathbf{3}, 4, \mathbf{5}, 2, \mathbf{1}, 6\right]       & (1, 3, 5) & (4, 2) \\
 \left[3, \mathbf{4}, \mathbf{5}, \mathbf{2}, 1, 6\right]       & (2, 3, 4) & (2, 1) \\
 \left[3, \mathbf{4}, \mathbf{5}, 2, \mathbf{1}, 6\right]       & (2, 3, 5) & (3, 1) \\  \hline
\end{array}
\end{equation*}
Thus, under the width system $w$ the instance $(2, 3, 4)$ is the minimal $[231]$-instance; it is also the only locally minimal $[231]$-instance.

In the permutation $y=[1, 4, 8, 5, 2, 7, 6, 3]$, we have the following instances and widths of the pattern $[231]$:
\begin{equation*}
\begin{array}[b]{|c c c|}
  \hline          &         P     &  w(P)  \\ \hline
 \left[ 1, \mathbf{4}, \mathbf{8}, 5, \mathbf{2}, 7, 6, 3 \right]  & (2, 3, 5) & (3, 1) \\
 \left[ 1, \mathbf{4}, \mathbf{8}, 5, 2, 7, 6, \mathbf{3} \right]  & (2, 3, 8) & (6, 1) \\
 \left[ 1, \mathbf{4}, 8, \mathbf{5}, \mathbf{2}, 7, 6, 3 \right]  & (2, 4, 5) & (3, 2) \\
 \left[ 1, \mathbf{4}, 8, \mathbf{5}, 2, 7, 6, \mathbf{3} \right]  & (2, 4, 8) & (6, 2) \\
 \left[ 1, \mathbf{4}, 8, 5, 2, \mathbf{7}, 6, \mathbf{3} \right]  & (2, 6, 8) & (6, 4) \\
 \left[ 1, \mathbf{4}, 8, 5, 2, 7, \mathbf{6}, \mathbf{3} \right]  & (2, 7, 8) & (6, 5) \\
 \left[ 1, 4, 8, \mathbf{5}, 2, \mathbf{7}, 6, \mathbf{3} \right]  & (4, 6, 8) & (4, 2) \\
 \left[ 1, 4, 8, \mathbf{5}, 2, 7, \mathbf{6}, \mathbf{3} \right]  & (4, 7, 8) & (4, 3) \\  \hline
\end{array}
\end{equation*}
Here, the instance $(2, 3, 5)$ is minimal under $w$.  Additionally, the instance $(4, 6, 8)$ is locally minimal, since it is the minimal instance of $[231]$ in the partial permutation 
$\left[ \mathbf{5}, 2, \mathbf{7}, 6, \mathbf{3} \right]$.
\end{example}

For certain width systems, minimality provides a natural factorization of $x$ over $\sigma$.

\begin{example}
\label{ex:231bountiful}

We consider the width system for the pattern $[231]$ depicted in Figure~\ref{fig.minimal231}.

\begin{figure}
  \begin{center}
  \includegraphics[scale=1]{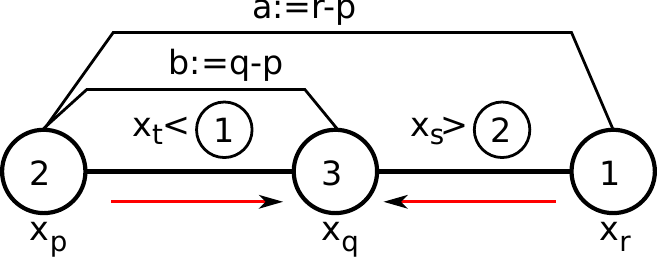}
  \end{center}
  \caption{A diagram of a minimal $[231]$ pattern.  The circled numbers represent elements $(x_p, x_q, x_r)$ filling the roles of the pattern; the widths are denoted $a$ and $b$, and the restrictions on $x_t$ with $p<t<q$ and $x_s$ with $s<q<r$ implied by minimality of the pair $(a,b)$ are also recorded.  The red arrows record the fact that shifting the end elements towards the center using a sequence of simple transpositions reduces the length of the permutation.}
  \label{fig.minimal231}
\end{figure}

Let $x=[x_1,x_2,\ldots,x_N] \in S_N$ containing a $[231]$-pattern, and let $(p<q<r)$ be the indices of a minimal-width $[231]$-pattern in $x$ under the width system $w=(r-p, q-p)$.  (So $x_r<x_p<x_q$.)

Minimality of the total width $(r-p)$ implies that for every $s$ with $q<s<r$, we have $x_s>x_p(>x_r)$, as otherwise $(x_p, x_q, x_s)$ would be a $[231]$-pattern of smaller width.  Then multiplying $x$ on the right by $u_1=s_{r-1}s_{r-2}\ldots s_{q+1}$ yields a permutation of length $\len(x)-(r-q-1)$, with 
\[
xu_1=[x_1,\ldots,x_p,\ldots,x_q,x_r,x_{q+1}\ldots,x_N].
\]

Minimality of the inner width $(q-p)$ implies that for every $t$ with $p<t<q$, then $x_t<x_r$.  (If $x_p<x_t<x_q$, then $(x_t, x_q, x_r)$ would form a $[231]$-pattern of lower width.  If $x_p>x_t$, then $q$ was not chosen minimally.)  Then multiplying $xu_1$ on the right by $u_2=s_ps_{p+1}\ldots s_{q-2}$ yields a permutation of length $\len(xu_1)-(q-p-1) = \len(x) - r + p + 2)$.  This permutation is:
\[
xu_1u_2=[x_1,\ldots,x_{q-1},x_p,x_q,x_r,x_{q+1}\ldots,x_N].
\]

Since $[x_p, x_q, x_r]$ form a $[231]$-pattern, we may further reduce the length of this permutation by multiplying on the right by $s_{q}s_{q-1}$.  The resulting permutation has no right descents in the set $J:=\{q-1, q\}$.

We then set $y=xu_1u_2s_{q}s_{q-1}$, $\sigma'=s_{q-1}s_{q}$, and $z=(u_1u_2)^{-1}$.  Notice that $z$ has no left descents in $\{q-1, q\}$ by construction, since it preserved the left-to-right order of $x_p, x_q$ and $x_r$.   Then $x=y \sigma' z$ is a factorization of $x$ over $\sigma$.
\end{example}

One may use a similar system of minimal widths to show that any permutation containing a $[321]$-pattern contains a braid, replicating a result of Billey, Jockusch, and Stanley~\cite{BilleyJockuschStanley.1993}.  The corresponding system of widths is depicted in Figure~\ref{fig.minimal321}.
\begin{figure}
  \begin{center}
  \includegraphics[scale=1]{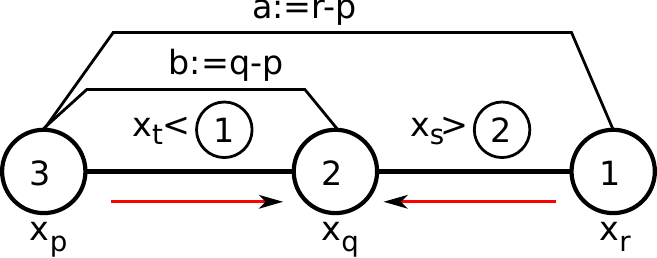}
  \end{center}
  \caption{A diagram of a left-minimal $[321]$ pattern, labeled analogously to the labeling in Figure~\ref{fig.minimal231}.}
  \label{fig.minimal321}
\end{figure}

\begin{definition}
Let $\sigma$ be a permutation with a width system.  The width system is \textbf{bountiful} if for any $x$ containing a locally minimal $\sigma$ at positions $(p_1, \ldots, p_k)$, any $x_t$ with $p_i<t<p_{i+1}$ has either $x_t<x_{p_k}$ for all $p_k<t$ or $x_t>x_{p_k}$ for all $p_k>t$.
\end{definition}

\begin{proposition}
If a pattern $\sigma$ admits a bountiful width system, then any $x$ containing $\sigma$ factorizes over $\sigma$.
\end{proposition}
\begin{proof}
By definition, any $x_t$ with $p_i<t<p_{i+1}$ has either $x_t<x_{p_k}$ for all $p_k<t$ or $x_t>x_{p_k}$ for all $p_k>t$.  Then using methods exactly as in Example~\ref{ex:231bountiful}, we may vacate the elements $x_t$ by multiplying on the right by simple transpositions, moving ``small'' $x_t$ out to the left and moving ``large'' $x_t$ out to the right.  This brings the minimal instance of the pattern $\sigma$ together into adjacent positions $(j, j+1, \ldots, j+k)$, while simultaneously creating a reduced word for the right factor $z$ in the factorization.  Then we set $J=\{j, j+1, \ldots, j+k-1\}$, and let $\sigma'$ be the $J$-shift of $\sigma$.  Set $y=x z^{-1} \sigma'^{-1}$.  Then by construction $x=y \sigma' z$ is a factorization of $x$ over $\sigma$.
\end{proof}

Thus, establishing bountiful width systems allows the direct factorization of $x$ containing $\sigma$ as an element of $W^J \sigma' \leftexp{J}{W}$.

\begin{problem}
Characterize the patterns which admit bountiful width systems.
\end{problem}

\begin{example}
\label{ex:123broken}
The permutation $x = [1324] = s_2$ contains a $[123]$-pattern, but does not factor over $[123]$.  To factor over $[123]$, we have $x\in W^J 1_J \leftexp{J}{W}$, with $J=\{1,2\}$ or $J=\{2,3\}$.  Both choices for $J$ contain $2$, so it is impossible to write $x$ as such a product.
\end{example}

\begin{proposition}
\label{prop:s2widthSystems}
Both patterns in $S_2$ admit bountiful width systems.
\end{proposition}
\begin{proof}
Any minimal $[12]$- or $[21]$-pattern must be adjacent, and so the conditions for a bountiful width system hold vacuously.
\end{proof}

\begin{proposition}
\label{prop:s3widthSystems}
All of the patterns in $S_3$ except $[123]$ admit a bountiful width system, as depicted in Figure~\ref{fig.s3widthSystems}.
\end{proposition}
\begin{figure}
  \begin{center}
  \includegraphics[scale=1]{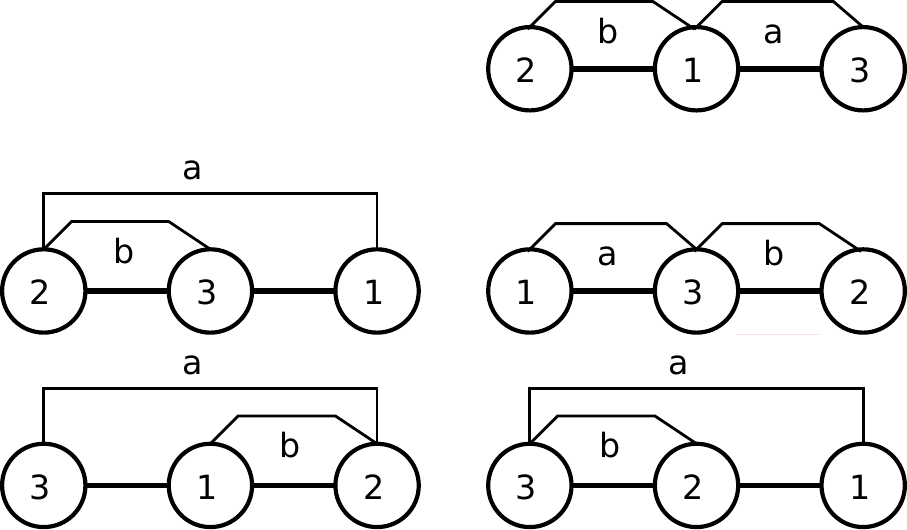}
  \end{center}
  \caption{Diagrams of bountiful width systems for the five patterns in $S_3$ which admit bountiful width systems.}
  \label{fig.s3widthSystems}
\end{figure}

\begin{proof}
A bountiful width systems has already been provided for the pattern $[231]$.  We only provide the details of the proof that the $[213]$ pattern is bountiful, as the proofs that the width systems for the patterns $[132]$, $[312]$ and $[321]$ are bountiful are analogous.

Let $x\in S_N$ contain a $[213]$ pattern at positions $(x_p, x_q, x_r)$, and choose the width system $(a,b)=(r-q, q-p)$.  

Suppose that $(x_p, x_q, x_r)$ is lexicographically minimal in this width system, and consider $x_t$ with $p<t<q$ and $x_s$ with $q<s<r$. Then $a=1$:
\begin{itemize}
\item If $x_s<x_p$, then $(x_p, x_s, x_r)$ is a $[213]$ pattern with $a$ smaller.
\item If $x_p<x_s$, then $(x_p, x_q, x_s)$ is a $[213]$ pattern with $a$ smaller.
\end{itemize}
Thus, we must have $r-q=1$.

Since $b$ is minimal, we must also have that $x_t>x_q$ or $x_t<x_p$ for every $t$ with $p<t<q$.  This completes the proof that the width system is bountiful.
\end{proof}

\begin{proposition}
\label{prop:widthSystemExtend}
Let $\sigma$ be a pattern in $S_{K-1}$ with a bountiful width system, and let $\sigma_+ = [K, \sigma_1, \ldots, \sigma_{K-1}]$. Then $\sigma_+$ admits a bountiful width system.  

Similarly, let $\sigma_-= [\sigma_1+1, \ldots, \sigma_{K-1}+1, 1]$.  Then $\sigma_-$ admits a bountiful width system.
\end{proposition}
\begin{proof}
Let $w=(w_1, w_2, \ldots, w_{k-2})$ be a bountiful width system on $\sigma$ (so $w_i$ is the difference between indices of an instance of $\sigma$ in a given permutation).  Let $x$ contain $\sigma_+$ in positions 
$( x_p, \ldots, x_q )$.  For $\sigma_+$, we show that the width system 
$w_+=(w_1, w_2, \ldots, w_{k-2}, q-p)$ is bountiful, where $w_i$ measures widths of elements in $\sigma$ as in $w$.

Consider a $\sigma_+$-pattern in a permutation $x$ that is minimal under the width system $w_+$, appearing at indices given by the tuple $p:=(i_1, \ldots, i_{k+1})$.  Then $x$ contains a $\sigma$-pattern at positions $(i_2, \ldots, i_{k+1})$.  This pattern may not be minimal under $w$ but, by the choice of width system, is as close as possible to being $w$-minimal, in the following sense.

We examine two cases.
\begin{itemize}
\item If there are no indices $t$ with $i_2<t<i_{k+1}$ such that $x_t>x_{i_1}$, then $\sigma$ must be $w$-minimal on the range $i_2, \ldots, i_{k+1}$.  (Otherwise, a $w$-minimal $\sigma$-pattern in that space would extend to a pattern that was less than $p$ in the $w_+$ width system.)  Then bountifulness of the $\sigma$ pattern ensures that for any $t$ with $i_j<t<i_{j+1}$ with $j\geq 2$; then $x_t<x_{i_k}$ for all $i_k<t$ or $x_t>x_{i_k}$ for all $i_k>t$.  (The ``small'' elements are still smaller than the ``large'' element $x_{i_1}$.)



\item On the other hand, if there exist some $t$ with $i_2<t<i_{k+1}$ such that $x_t>x_{i_1}$, we may move these $x_t$ out of the $\sigma$ pattern to the right by a sequence of simple transpositions, each decreasing the length of the permutation by one.  Let $u$ be the product of this sequence of simple transpositions.  Then $xu$ fulfills the previous case.  Each of the $x_t$ were larger than all pattern elements to the right, so we see that $\sigma_+$ fulfills the requirements of a bountiful pattern.

\end{itemize}

The proof that $\sigma_-$ admits a bountiful width system is similar.
\end{proof}

\begin{corollary}
\label{cor:bountifulPerms}
Let $\sigma \in S_K$ be a permutation pattern, where the length of $\sigma$ is at most one less than the length of the long element in $S_K$.  Then $\sigma$ admits a bountiful width system.
\end{corollary}
\begin{proof}
This follows inductively from Proposition~\ref{prop:widthSystemExtend}, and the fact that the patterns $[12]$ and $[21]$ both admit bountiful width systems.
\end{proof}
\begin{figure}
  \begin{center}
  \includegraphics[scale=1]{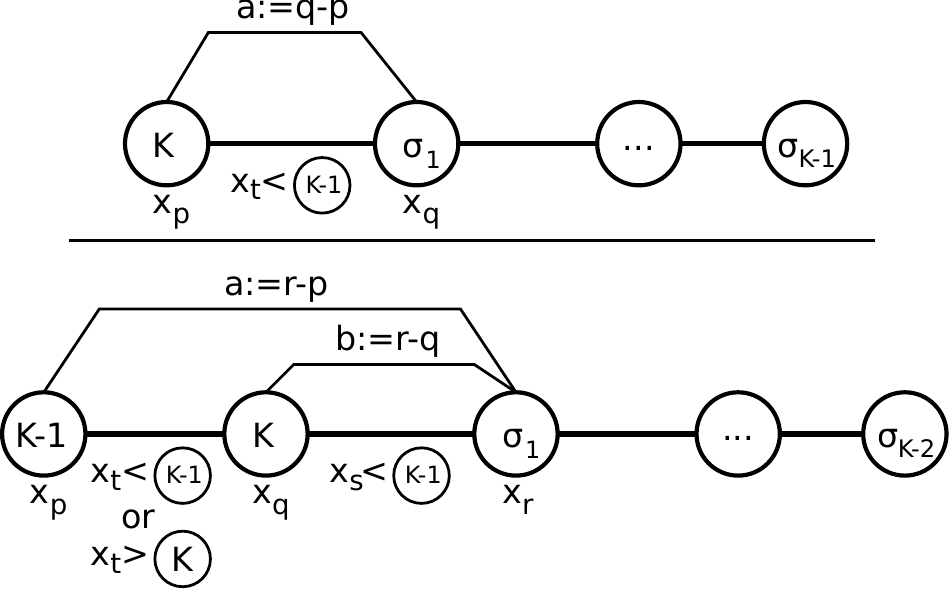}
  \end{center}
  \caption{Diagram of extensions of a bountiful width system $w$ by the additional widths $a$ or $(a, b)$, as described in the proofs of Propositions~\ref{prop:widthSystemExtend} and~\ref{prop:widthSystemExtend2}. }
  \label{fig.widthSystemExtend}
\end{figure}

\begin{proposition}
\label{prop:widthSystemExtend2}
Let $\sigma$ be a pattern in $S_{K-2}$ with a bountiful width system, and let $\sigma_{++} = [K-1, K, \sigma_1, \ldots, \sigma_{K-1}]$. Then $\sigma_{++}$ admits a bountiful width system.  

Similarly, let $\sigma_{--}= [\sigma_1+2, \ldots, \sigma_{K-2}+2, 1, 2]$.  Then $\sigma_{--}$ admits a bountiful width system.
\end{proposition}
\begin{proof}
The proof of this proposition closely mirrors the proof of Proposition~\ref{prop:widthSystemExtend}.  Let $w=(w_1, w_2, \ldots, w_{k-2})$ a bountiful width system on $\sigma$.  Let $x$ contain $\sigma_{++}$ in positions 
$( x_p, x_r, x_s, \ldots, x_q )$.  For $\sigma_{++}$, we claim that the width system 
$w_{++}=(w_1, w_2, \ldots, w_{k-2}, q-p, s-r)$ is bountiful, where $w_i$ measures widths of elements in $\sigma$ as in $w$.  (The width system $w_{++}$ is depicted in Figure~\ref{fig.widthSystemExtend}.)

Again, local minimality of $\sigma$ ensures that all $x_t$ with $s<t<q$ with $x_t$ not in the instance of $\sigma_{++}$ are either smaller than all pattern elements to the left of $x_t$, or larger than all pattern elements to the right of $x_t$.  The choice of $w_{++}$ ensures that all $x_t$ with $p<t<r$ are either less than $x_p$ or larger than $x_r$, and that all $x_t$ with $r<t<s$ are less than $x_p$.  Then $w_{++}$ is bountiful.

The proof that $\sigma_{--}$ is bountiful is analogous.
\end{proof}

\subsection{Further Directions}
Preliminary investigation suggests that patterns admitting a bountiful width system are somewhat rare, though there are more than those described by Corollary~\ref{cor:bountifulPerms}.  Weakening the definition of a factorization over a permutation may provide an additional avenue of investigation, though.

\begin{definition}
A permutation $x\in W=S_N$ \textbf{left-factorizes} over a pattern $\sigma\in S_K$ if $x=y\sigma' z$ with:
\begin{itemize}
\item $\sigma' \in W_J$, with $J=\{j, j+1, \ldots, j+k\}$ and $\sigma'$ containing a $\sigma$-pattern,
\item $y \in W^J$,
\item $\len(x)=\len(y) + \len(\sigma) + \len(z)$.
\end{itemize}
\end{definition}

This definition drops the requirement that $z \in \leftexp{J}{W}$.  This definition may be too weak, though, since one can show that any permutation containing the pattern $[K, K-1, \ldots, 1]$ left-factors over every pattern in $S_K$.

On the other hand, consider Example~\ref{ex:123broken}.  The permutation $x=[1, 3, 2, 4]=s_2$ admits a factorization $S^{\{1,3\}} 1_{\{1,3\}} \leftexp{\{1,3\}}{S}$, and the element $1_{\{1,3\}}$ contains a $[123]$-pattern.  Allowing factorizations over arbitrary subgroups -- and obtaining a combinatorial characterization of these factorizations -- may provide a way forward.

\begin{problem}
Find a general characterization of pattern containment in terms of factorizations of a permutation.
\end{problem}

\section{Pattern Avoidance and the $\NDPF$ Quotient}
\label{sec:ndpfPattAvoid}

In this section, we consider certain quotients of the $0$-Hecke monoid of the symmetric group, and relate the fibers of the quotient to pattern-avoidance.  The $0$-Hecke monoid $H_0(S_N)$ is defined in Definition~\ref{ssec:zeroHeckeDefinition}, and the Non-decreasing Parking Function $\NDPF_N$ quotient is discussed in Section~\ref{ssec:ndpf}, in its guise as the the monoid of order-preserving regressive functions on a chain.

\begin{definition}
For $x\in H_0(S_N)$, we say $x$ {\bf contains a braid} if some reduced word for $x$ contains a contiguous subword $\pi_i \pi_{i+1} \pi_i$.  

The permutation $x$ contains an {\bf unmatched ascent} if some reduced word for $x$ contains a contiguous subword $\pi_i \pi_{i+1}$ that is not part of a braid.  More precisely, if inserting a $\pi_i$ directly after the $\pi_i \pi_{i+1}$ increases the length of $x$, then $x$ contains an unmatched ascent.  Equivalently, $x$ may be factorized as $x=y \pi_i \pi_{i+1} z$, where $y$ has no right descents in $\{i, i+1\}$, and $z$ has no left descents in $\{i, i+1\}$, and $\len(x)=\len(y)+2+\len(z)$.

An {\bf unmatched descent} is analogously defined as a contiguous subword $\pi_{i+1}\pi_i$ such that insertion of a $\pi_i$ immediately before this subword increases the length of $x$.  Equivalently, $x$ may be factorized as $x=y \pi_{i+1}\pi_i z$, where $y$ has no right descents in $\{i, i+1\}$, and $z$ has no left descents in $\{i, i+1\}$, and $\len(x)=\len(y)+2+\len(z)$.
\end{definition}

\begin{lemma}
\label{lemma:231unmatched}
For $x\in S_N$, $x$ contains a $[231]$-pattern if and only if $x$ has an unmatched ascent.  Likewise, $x$ contains a $[312]$-pattern if and only if $x$ has an unmatched descent.
\end{lemma}
\begin{proof}
This is a straightforward application of the bountiful width system for the patterns $[231]$ and $[312]$.  The resulting factorization contains an unmatched ascent (resp., descent).
\end{proof}

This process of inserting an $s_i$ can be made more precise in the symmetric group setting: suppose $s_{j_1}\ldots s_{i}s_{i+1}\ldots s_{j_k}$ is a reduced expression for $x \in S_N$.  Then write $x=x_1 s_{i}s_{i+1} x_2$.  To insert $s_i$, multiply $x$ on the right by $x_2^{-1}s_ix_2$.  As such, this insertion can be realized as multiplication by some reflection.

This insertion is generally not a valid operation in $H_0(S_N)$, since inverses do not exist.  However, the operation does make sense in the $\NDPF$ setting: the $\NDPF$ relation simply allows one to exchange a braid for an unmatched ascent or vice-versa.

\begin{theorem}
\label{thm:ndpfFibers231}
Each fiber of the map $\phi: H_0(S_N)\to \NDPF_N$ contains a unique $[321]$-avoiding element of minimal length and a unique $[231]$-avoiding element of maximal length.
\end{theorem}

\begin{proof}
The first part of the theorem follows directly from a result of Billey, Jockusch, and Stanley~\cite{BilleyJockuschStanley.1993}, which states that a symmetric group element contains a braid if and only if the corresponding permutation contains a $[321]$.  Alternatively, one can use the width system for $[321]$ established in Proposition~\ref{prop:s3widthSystems} to obtain a factorization including a braid.  Then for any $x$ in the fiber of $\phi$, one can remove braids obtained from minimal-width $[321]$-patterns using the $\NDPF$ relation and obtain a $[321]$-avoiding element.  Each application of the $\NDPF$-relation reduces the length of the permutation by one, so this process must eventually terminate in a $[321]$-avoiding element.  Uniqueness follows since there are exactly $C_N$ $[321]$-avoiding elements in $S_N$, where $C_N$ is the $N$th Catalan number, and are thus in bijection with elements of $\NDPF_N$.

For the second part, we use the bountiful $[231]$ width system established in Example~\ref{ex:231bountiful}.  Let $x$ contain a $[231]$-pattern.  The width system allows us to write a factorization $x=y \pi_i \pi_{i+1} z$, where $y$ has no right descents in $\{i, i+1\}$ and $z$ has no left descents in $\{i, i+1\}$.  Then we may apply the $\NDPF$ relation to insert a $\pi_i$, turning the $[231]$-pattern into a $[321]$ pattern, and increasing the length of $x$ by one.  Since we are in a finite symmetric group, there is an upper bound on the length one may obtain by this process, and so the process must terminate with a $[231]$-avoiding element.  Recall that $[231]$-avoiding permutations are also counted by the Catalan numbers~\cite{knuth.TAOCP1}, and apply the same reasoning as above to complete the theorem.
\end{proof}

Recall that the \emph{right action} of $S_N$ acts on positions.  A permutation $y$ has a \emph{right descent at position i} if the two consecutive elements $y_i,y_{i+1}$ are out of order in one-line notation.  Then multiplying on the right by $s_i$ puts these two positions back in order and reduces the length of $y$ by one.  Likewise, if $y$ does not have a right descent at $i$, multiplying by $s_i$ increases the length by one.

\begin{example}[Fibers of the $\NDPF$ quotient]
For $S_4$, the fibers of the $\NDPF$ quotient can be found in Figure~\ref{fig.s4ndpfFibers}.  

\begin{figure}
  \begin{center}
  \includegraphics[scale=1]{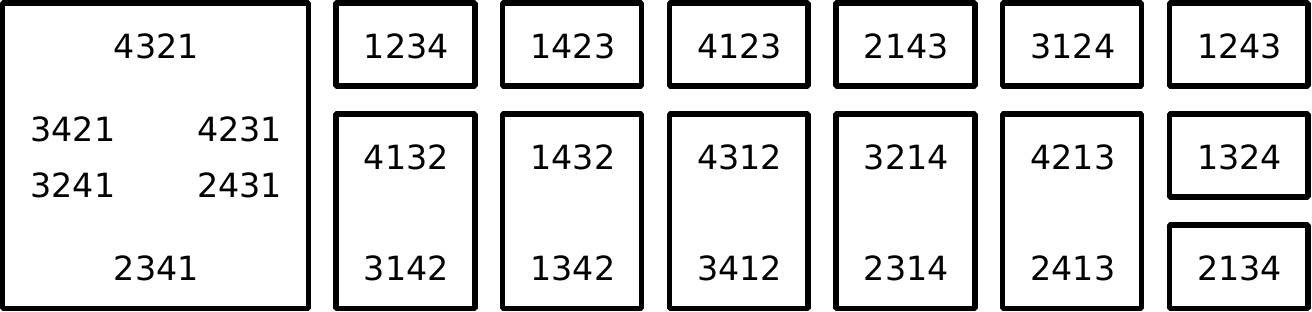}
  \end{center}
  \caption{Fibers of the $\NDPF$ quotient for $H_0(S_4)$.}
  \label{fig.s4ndpfFibers}
\end{figure}
As a larger example, let $\sigma=[3,6,4,5,7,2,1]\in S_7$.  For Lemma~\ref{lemma:231unmatched}, we find minimal-width $[231]$-patterns, with the element corresponding to the $3$ chosen as far to the left as possible.  (The subsequence $(5, 7, 2)$ of $\sigma$ is such a minimal $[231]$-pattern.)  Then applying the transformation $[231]\rightarrow [321]$ on that instance of the pattern preserves the fiber of the $\NDPF$ quotient, and increases the length of the permutation by $1$.  By sequentially removing eight such minimal $[231]$-patterns, one obtains the long element in $S_7$, which is $[231]$-avoiding.  The fiber containing the long element also contains a $[321]$-avoiding element $[2,3,4,5,6,7,1]$, which has length $6$, and is the shortest element in its fiber.
\end{example}

We now fix bountiful width system for $[231]$- and $[321]$-patterns, which we will use for the remainder of this section.
\begin{definition}
Let $x\in S_N$, $x=[x_1, \ldots, x_N]$ in one-line notation, and consider all $[231]$-patterns $(x_p,x_q,x_r)$ in $x$.  The \textbf{width} of a $[231]$-pattern $(x_p,x_q,x_r)$ is the pair $(r-p, q-p)$.  The pattern is a \textbf{minimally chosen $[231]$-pattern} if the width is lexicographically minimal amongst all $[231]$-patterns in $x$.

On the other hand, call a $[321]$-pattern $(x_p,x_q,x_r)$ \textbf{left minimal} if for all $t$ with $p<t<q$, $x_t<x_r$, and for all $s$ with $q<s<r$, $x_s>x_q$.
\end{definition}

The following is a direct result of the proof of Lemma~\ref{lemma:231unmatched}.
\begin{corollary}
\label{cor:fiber231}
Let $x\in S_N$.  Let $(x_p,x_q,x_r)$ be a minimally chosen $[231]$-pattern in $x$.  Then the permutation 
\[
[x_1, \ldots, x_{p-1}, x_q, x_{p+1}, \ldots, x_{q-1}, x_p, x_{q+1}, \ldots, x_r, \ldots, x_N],
\]
obtained by applying the transposition $t_{p,q}$, is in the same $\NDPF$-fiber as $x$.  The result of applying this transposition is a left-minimal $[321]$-pattern.
\end{corollary}

\subsection{Involution}
\label{subsec:involution}

Let $\Psi$ be the involution on the symmetric group induced by conjugation by the longest word.  Then $\Psi$ acts on the generators by sending $s_i \to s_{N-i}$.  This descends to an isomorphism of $H_0(S_N)$ by exchanging the generators in the same way: $\pi_i \to \pi_{N-i}$.

We can thus obtain a second map from $H_0(S_N)\to \NDPF_N$ by pre-composing with $\Psi$.  This has the effect of changing the $\NDPF$ relation to a statement about unmatched \emph{descents} instead of unmatched ascents.  Then applying the $\NDPF$ relation allows one to exchange braids for unmatched descents and vice-versa, giving the following theorem.

\begin{theorem}
\label{thm:ndpfFibers312}
Each fiber of the map $\phi \circ \Psi: H_0(S_N)\to \NDPF_N$ contains a unique $[321]$-avoiding element for minimal length and a unique $[312]$-avoiding element of maximal length.
\end{theorem}

The proof is exactly the mirror of the proof in previous section.

We fix bountiful width system for $[312]$-patterns, and a second bountiful width system for $[321]$-patterns, which we will use for the remainder of this section.
\begin{definition}
Let $x\in S_N$, $x=[x_1, \ldots, x_N]$ in one-line notation, and consider all $[312]$-patterns $(x_p,x_q,x_r)$ in $x$.  The \textbf{width} of a $[312]$-pattern $(x_p,x_q,x_r)$ is the pair $(r-p, r-q)$.  The pattern is a \textbf{minimally chosen $[312]$-pattern} if the width is lexicographically minimal amongst all $[312]$-patterns in $x$.

Likewise, call a $[321]$-pattern $(x_p,x_q,x_r)$ \textbf{right minimal} if the \textbf{right width} $(p-r, r-q)$ is lexicographically minimal amongst all $[321]$-patterns in $x$.
On the other hand, call a $[321]$-pattern $(x_p,x_q,x_r)$ \textbf{right minimal} if for all $t$ with $p<t<q$, $x_t<x_q$, and for all $s$ with $q<s<r$, $x_s>x_p$.
\end{definition}

\begin{corollary}
\label{cor:fiber312}
Let $x\in S_N$.  Let $(x_p,x_q,x_r)$ be a minimally chosen $[312]$-pattern in $x$.  Then the permutation 
\[
[x_1, \ldots, x_{p-1}, x_q, x_{p+1}, \ldots, x_{q-1}, x_p, x_{q+1}, \ldots, x_r, \ldots, x_N],
\]
obtained by applying the transposition $t_{p,q}$, is in the same $\NDPF\circ \Psi$-fiber as $x$.  The result of applying this transposition is a right-minimal $[321]$-pattern.
\end{corollary}

\section{Type B $\NDPF$ and $[4321]$-Avoidance}
\label{sec:bndpfPattAvoid}

In this section, we establish a monoid morphism of $H_0(S_N)$ whose fibers each contain a unique $[4321]$-avoiding permutation.  To motivate this map, we begin with a discussion of Non-Decreasing Parking Functions of Type $B$.

The Weyl Group of Type $B$ may be identified with the \textbf{signed symmetric group} $S_N^B$, which is discussed (for example) in~\cite{Bjorner_Brenti.2005}.  Combinatorially, $S_N^B$ may be understood as a group permuting a collection of $N$ labeled coins, each of which can be flipped to heads or tails.  The size of $S_N^B$ is thus $2^NN!$.   A minimal set of generators of this group are exactly the simple transpositions $\{t_i\mid i \in \{1, \ldots, N-1\}\}$ interchanging the coins labeled $i$ and $i+1$, along with an extra generator $t_N$ which flips the last coin.  

The group $S_N^B$ can be embedded into $S_{2N}$ by identifying the $t_i$ with $s_is_{2N-i}$ for each $i \in \{1, \ldots, N-1\}$, and $t_N$ with $s_N$.

\begin{definition}
The {\bf Type B Non-Decreasing Parking Functions} $\BNDPF_N$ are the elements of the submonoid of $\NDPF_{2N}$ generated by the collection $\mu_i := \pi_i\pi_{2N-i}$ for $i$ in the set $\{1, \ldots, N\}$.
\end{definition}

Note that $\mu_N = \pi_N^2 = \pi_N $.

The number of $\BNDPF_N$ has been explicitly computed up to $N=9$, though a proof for a general enumeration has proven elusive, in the absence of a more conceptual description of the full set of functions generated thusly.  The sequence obtained (starting with the $0$-th term) is
\[
	( 1, 2, 7, 33, 183, 1118, 7281, 49626, 349999, 253507, \ldots ),
\]
which agrees with the sequence 
\[
\sum_{j=0}^N \binom{N}{j}^2 C_j 
\]
so far as it has been computed.  This appears in Sloane's On-Line Encyclopedia of Integer Sequences as sequence $A086618$~\cite{Sloane}, and was first noticed by Hivert and Thi\'ery~\cite{Hivert.Thiery.HeckeGroup.2007}.

\begin{conjecture}
\[ 
| \BNDPF_N | = \sum_{j=0}^N \binom{N}{j}^2 C_j.
\]
\end{conjecture}

Let $X$ be some object (group, monoid, algebra) defined by generators $S$ and relations $R$.  Recall that a \emph{parabolic subobject} $X_J$ is generated by a subset $J$ of the set $S$ of simple generators, retaining the same relations $R$ as the original object.  Let $\BNDPF_{N,\hat{N}}$ denote the parabolic submonoid of of $\BNDPF_N$ retaining all generators but $\mu_N$.

Consider the embedding of $\BNDPF_{N,\hat{N}}$ in $\NDPF_{2N}$.  Then a reduced word for an element of $\BNDPF_{N,\hat{N}}$ can be separated into a pairing of $\NDPF_N$ elements as follows:
\begin{align}
\mu_{i_1}\mu_{i_2}\ldots\mu_{i_k} &=& \pi_{i_1}\pi_{2N-i_1}\pi_{2N-i_2}\pi_{i_2}\ldots\pi_{i_k}\pi_{2N-i_k} \\
&=& \pi_{i_1}\pi_{i_2}\ldots\pi_{i_k}\pi_{2N-i_1}\pi_{2N-i_2}\ldots\pi_{2N-i_k}
\end{align}

In particular, one can take any element $x\in H_0(S_N)$ and associate it to the pair:
\[
\omega(x):=(\phi(x), \phi\circ\Psi(x)),
\] 
recalling that $\Psi$ is the Dynkin automorphism on $H_0(S_N)$, described in Section~\ref{subsec:involution}.

Given the results of the earlier section, one naturally asks about the fiber of $\omega$.  It is easy to do some computations and see that the situation is not quite so nice as before.  In $H_0(S_4)$ the only fiber with order greater than one contains the elements $[4321]$ and $[4231]$.  Notice what happens here: $[4231]$ contains both a $[231]$-pattern and a $[312]$-pattern, which is straightened into two $[321]$-patterns.   On the level of reduced words, two reduced words for $[4231]$ are $((3,2,1,2,3))=((1,2,3,2,1))$, one of which ends with the unmatched ascent $[2,3]$ while the other ends with the unmatched descent $[2,1]$.  Multiplying on the right by the simple transposition $s_2$ matches both of these simultaneously.

In fact, this is a perfectly general operation.  Let $x\in H_0(S_N)$.  For any minimally-chosen $[231]$-pattern in $x$, one can locate an unmatched ascent in $x$ that corresponds to the pattern.  Here the smaller element to the right remains fixed while the two ascending elements to the left are exchanged.  Then applying the $\NDPF$ relation to turn the $[231]$ into a $[321]$ preserves the fiber of $\phi$.  Likewise, one can turn a minimal $[312]$ into a $[321]$ and preserve the fiber of $\phi\circ\Psi(x)$.  Here the larger element to the left is fixed while the two ascending elements to the right are exchanged.  Hence, to preserve the fiber of $\omega$, one must find a pair of ascending elements with a large element to the left and a small element to the right: this is exactly a $[4231]$-pattern.  

One may make this more precise by defining a system of widths under which minimal $[4231]$-patterns contain a locally minimal $[231]$-pattern and a locally-minimal $[321]$-pattern.  The results of Section~\ref{sec:widthSystems} imply that this is possible.  Applying the $\NDPF$ relation, this becomes a $[4321]$. 

On the other hand, we can define a minimal $[4321]$-pattern by a tuple of widths analogous to the constructions of minimal $[231]$-patterns.  The construction of this tuple, and the constraints implied when the tuple is minimal, is depicted in Figure~\ref{fig.minimal4321}.  Such a minimal pattern may always be turned into a $[4231]$-pattern while preserving the fiber of $\omega$.

Let $x\in S_N$ and $P=(x_p, x_q, x_r, x_s)$ a $[4321]$-pattern in $x$.  For the remainder of this section, we fix the width system $(q-p, r-q, s-r)$, and use the same width system for $[4231]$-patterns.  One may check directly that this is a bountiful width system in both cases.

\begin{figure}
  \begin{center}
  \includegraphics[scale=1]{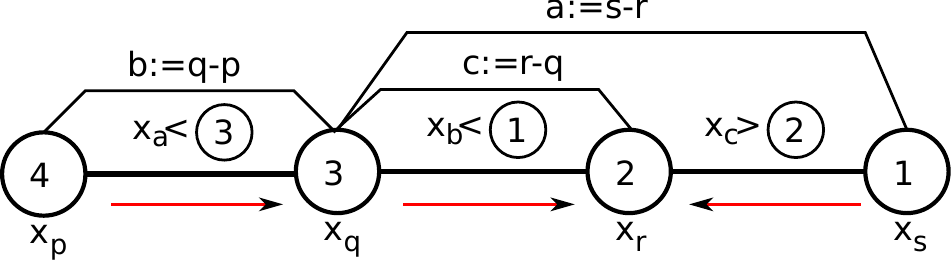}
  \end{center}
  \caption{A diagram of a minimal $[4321]$ pattern, labeled analogously to the labeling in Figure~\ref{fig.minimal231}.}
  \label{fig.minimal4321}
\end{figure}

\begin{lemma}
Let $x$ contain a minimal $[4321]$-pattern $P=(x_p, x_q, x_r, x_s)$, and let $x' = x t_{r,s}$, where $t_{r,s}$ is the transposition exchanging $x_r$ and $x_s$.  Then $\omega(x')=\omega(x)$.
\end{lemma}
\begin{proof}
Since the width system on $[4321]$-patterns is bountiful, we can factor $x = y x_J z$, with $\len(x)=\len(y)+ \len(x_J)+\len(z)$ where 
\[x_J = s_{s-2}s_{s-1} s_s s_{s-1} s_{s-2}s_{s-1}.\]
By the discussion above, the trailing $s_{s-1}$ in $x_J$ may be removed to simultaneously yield an unmatched ascent and an unmatched descent.  Then this removal preserves the fiber of both $\phi$ and $\Psi\circ \phi$, and thus also preserves the fiber of $\omega$.
\end{proof}

Note that there need not be a unique $[4231]$-avoiding element in a given fiber of $\omega$.  The first example of this behavior occurs in $N=7$, where there is a fiber consisting of $[5274163], [5472163],$ and $[5276143]$.  In this list, the first element is $[4321]$-avoiding, and the two latter elements are $[4231]$-avoiding.  In the first element, there are $[4231]$ patterns $[5241]$ and $[7463]$ which can be respectively straightened to yield the other two elements.  Notice that either transposition moves the 4 past the bounding element of the other $[4231]$-pattern, thus obstructing the second transposition.

\begin{theorem}
\label{thm:ndpfFibers4321}
Each fiber of $\omega$ contains a unique $[4321]$-avoiding element.
\end{theorem}

\begin{proof}
Given any element of $H_0(S_N)$, we have seen that we can preserve the fiber of $\omega$ by turning locally minimal $[4321]$-patterns into $[4231]$-patterns.  Each such operation reduces the length of the element being acted upon, and thus this can only be done so many times.  Furthermore, any minimal-length element in the fiber of $\omega$ will be $[4321]$-avoiding.  We claim that this element is unique.  

First, note that one can impose a partial order on the fiber of $\omega$ with $x$ covering $y$ if $x$ is obtained from $y$ by turning a locally minimal $[4321]$-pattern into a $[4231]$-pattern.  Then the partial order is obtained by taking the transitive closure of the covering relation.  Note that if $x$ covers $y$ then $x$ is longer than $y$.  The Hasse diagram of this poset is connected, since any element of the fiber can be obtained from another by a sequence of $\NDPF$ relations respecting both the fiber of $\phi$ and $\phi\circ\Psi(x)$.  

Let $x$ be an element of $H_0(S_N)$ containing (at least) two locally minimal $[4321]$-patterns, in positions $(x_a,x_b,x_c,x_d)$ and $(x_p,x_q,x_r,x_s)$, with $a<b<c<d, p<q<r<s$.  Then one can exchange $x_b$ with $x_c$ or $x_q$ with $x_r$ and preserve the fiber of $\omega$.  Let $y$ be the element obtained from exchanging $x_b$ with $x_c$, and $z$ obtained by exchanging $x_q$ with $x_r$.  Then we claim that there exists $w$ covered by both $y$ and $z$.  (In other words, the poset structure on each fiber is a meet semilattice.)

If the tuples $(a,b,c,d)$ and $(p,q,r,s)$ are disjoint, then the claim is clearly true.  Likewise, if $a=p$ and/or $d=s$ the claim holds.  A complete but perhaps unenlightening proof of the claim can be accomplished by showing that it holds for all $\BNDPF_{N,\hat{N}}$ with $N<8$, where every possible intermingling of the tuples with every possible ordering of the entries $x_.$ occurs at least once.  It is best to perform this check with a computer, given that there are 2761 elements in $\BNDPF_{7,\hat{7}}$, with $7! = 5040$ elements in the fibers, and indeed a computer check shows that the claim holds.  The code accomplishing this is provided below.

Let's look at a couple cases, though, to get a feeling for why this should be true.  Refer to the extremal elements at the edge of the $[4321]$ pattern as the ``boundary,'' and the elements to be transposed as the ``interior.''  The main cases are the following:

Case $c=r$: Just take the smaller of $x_s$ and $x_d$ to be the common right boundary for both patterns.

Case $c=q$: The problem for $[4231]$ patterns was that one could apply a transposition that obstructed the other transposition by moving one of the interior elements past its boundary.  But here, we have $x_d<x_c$ and $x_s<x_q=x_d$, so we can use $s$ as the boundary for both patterns, and the obstruction is averted.  In this case, though, the two transpositions generate six elements in the fiber, instead of four.  We can still find a common meet, though.  $[x_ax_bx_qx_rx_s]$ becomes $[x_ax_qx_bx_rx_s]$ and $[x_ax_bx_rx_qx_s]$, which both cover $[x_ax_rx_bx_qx_s]$, for example.

Case $r=d$ or $q=d$: Again, just take $s$ as a common boundary for the two patterns.

And so on.  Many cases are symmetric to the three considered above, and every interesting case is solved by changing the boundary of one of the patterns.

Now that every pair of elements have a common meet, we are almost done.  Suppose there exist two different $[4321]$-avoiding elements $A_1$ and $A_2$ in some fiber.  Then since the fiber is connected, we can find a minimal element $x$ where a branching occurred, so that $x$ covers both $y>A_1$ and $z>A_2$, and $x$ is of minimal length.  But if both $y$ and $z$ were obtainable from $x$, then there exists a $w$ of shorter length below them both.  Now $w$ sits above some $[4321]$-avoiding element, as well.  If $w>A_1$ but not $A_2$, then in fact a branching occurred at $z$, contradicting the minimality of $x$.  The same reasoning holds if $w>A_2$ but not $A_1$.  If $w$ is above both $A_1$ and $A_2$, then in fact $y$ was comparable to $A_2$ and $z$ was comparable to $A_1$, and there was not a branching at $x$ at all.
\end{proof}

\subsection{Code for Theorem~\ref{thm:ndpfFibers4321}.}
Here we provide code for checking the claim of Theorem~\ref{thm:ndpfFibers4321} that each fiber of $\omega$ contains a unique $[4321]$-avoiding element.  The code is written for the Sage computer algebra system, which has extensive built-in functions for combinatorics of permutations, including detecting the presence of permutation patterns.  

The code below constructs a directed graph (see the function \textbf{omegaFibers}) whose connected components are fibers of $\omega$.  The vertices of this graph are permutations, and the edges correspond to straightening locally-minimal $[4231]$-patterns into $[4321]$ patterns.  A component is `bad' if it does not contain exactly one $[4321]$-avoiding permutation.
\begin{verbatim}
def width4231(p):
    """
    This function returns the width of a [4231]-instance p.
    """
    return (p[1]-p[0], p[2]-p[1], p[3]-p[2]) 

def min4231(x):
    """
    This function takes a permutation x and finds all minimal-width
    4231-patterns in x, and returns them as a list.
    """
    P=x.pattern_positions([4,2,3,1])
    if P==[]:
        return None
    minimal=[P[0]]
    for i in [1..len(P)-1]:
        if width4231(P[i])<width4231(minimal[0]):
            minimal = [ P[i] ]
        else:
            if width4231(P[i])==width4231(minimal[0]):
                minimal.append(P[i])
    return minimal
    
def localMin4231(x):
    """
    This function finds all locally-minimal 4231-patterns in a 
    permutation x, and returns them as a list.
    """
    P=x.pattern_positions([4,2,3,1])
    if P==[]:
        return None
    localMin=[]
    for p in P:
        xp=Permutation(x[ p[0]:p[3]+1 ])
        qp=[i - p[0] for i in p]
        qmin=min4231(xp)
        if qp in qmin: localMin.append(p)
    return localMin

def omegaFibers(N):
    """
    Given N, this function builds a digraph whose vertices are given by
    permutations of N, and with an edge a->b whenever b is obtained
    from a by straightening a locally minimal 4231-pattern into a 
    4321-pattern. 
    The connected components of G are the fibers of the map omega.
    """
    S=Permutations(N)
    G=DiGraph()
    G.add_vertices(S.list())
    for x in S:
        if x.has_pattern([4,2,3,1]):
            # print x, localMin4231(x)
            #add edges to G for each locally minimal 4231.
            Q=localMin4231(x)
            for q in Q:
                y=Permutation((q[1]+1,q[2]+1))*x
                G.add_edge(x,y)
    return G
    
def headCount(G):    
    """
    This function takes the diGraph G produced by the omegaFibers 
    function, and finds any connected components with more than one
    4321-pattern.  It returns a list of all such connected components.
    """
    bad=[]
    for H in G.connected_components_subgraphs():
        total=0
        for a in H:
            if not a.has_pattern([4,3,2,1]): total+=1
        if total != 1:
            #prints if any fiber has more than one 4321-av elt
            print H, total
            bad.append(H)
    print "N =", N 
    print "\tTotal connected components: \t", count
    print "\tBad connected components: \t", len(bad), '\n'
    return bad
\end{verbatim}

As explained in Theorem~\ref{thm:ndpfFibers4321}, we should check that each fiber of $\omega$ contains a unique $[4321]$-avoiding element for each $N\leq 7$.  This is accomplished by running the following commands:
\begin{verbatim}
sage: for N in [1..7]:
sage:     G=omegaFibers(N)
sage:     HH=headCount(G)
\end{verbatim}
The output of this loop is as follows:
\begin{verbatim}
N = 1
	Total connected components: 	1
	Bad connected components: 	0 

N = 2
	Total connected components: 	2
	Bad connected components: 	0 

N = 3
	Total connected components: 	6
	Bad connected components: 	0 

N = 4
	Total connected components: 	23
	Bad connected components: 	0 

N = 5
	Total connected components: 	103
	Bad connected components: 	0 

N = 6
	Total connected components: 	513
	Bad connected components: 	0 

N = 7
	Total connected components: 	2761
	Bad connected components: 	0 
\end{verbatim} 
There are no bad components, and thus the theorem holds.  

The sequence $(1, 2, 6, 23, 103, 513, 2761)$ is the beginning of the sequence counting $[4321]$-avoiding permutations.  This sequence also counts $[1234]$-avoiding permutations (reversing a $[1234]$-avoiding permutation yields a $[4321]$-avoiding permutation, and \textit{vice versa}), and is listed in that context in Sloane's On-Line Encyclopedia of Integer Sequences (sequence $A005802$)~\cite{Sloane}.

The author executed this code on a computer with a 900-mhz Intel Celeron processor (blazingly fast by 1995 standards) and 2 gigabytes of RAM.  On this machine, the $N=6$ case took 3.86 seconds of CPU time, and the $N=7$ case took just over one minute (62.06s) of CPU time.  The $N=8$ case (which is unnecessary to the proof) correctly returns 15767 connected components, none of which are bad, in 1117.24 seconds (or 18.6 minutes).

\section{Affine $\NDPF$ and Affine $[321]$-Avoidance}
\label{sec:affNdpfPattAvoid}

The affine symmetric group is the Weyl group of type $A_N^{(1)}$, whose Dynkin diagram is given by a cycle with $N$ nodes.  All subscripts on generators for type $A_N^{(1)}$ in this section will be considered $(\text{mod } N)$.  A combinatorial realization of this Weyl group is given below.

\begin{definition}
\label{def:affSn}
The \textbf{affine symmetric group} $\ASn_N$ is the set of bijections $\sigma: \ZZ \rightarrow \ZZ$ satisfying:
\begin{itemize}
    \item Skew-Periodicity: $\sigma(i+N) = \sigma(i) + N$, and
    \item Sum Rule: $\sum_{i=1}^N \sigma(i) = \binom{N+1}{2}$.
\end{itemize}
\end{definition}

We will often denote elements of $\ASn_N$ in the \textbf{window notation}, which is a one-line notation where we only write $(\sigma(1), \sigma(2), \ldots, \sigma(N))$.  Due to the skew-periodicity restriction, writing the window notation for $\sigma$ specifies $\sigma$ on all of $\ZZ$.  

The generators $s_i$ of $\ASn_N$ are indexed by the set $I=\{0, 1, \ldots, N-1\}$, and $s_i$ acts by exchanging $j$ and $j+1$ for all $j \equiv i (\text{mod } N)$.  These satisfy the relations:
\begin{itemize}
    \item Reflection: $s_i^2 = 1$,
    \item Commutation: $s_j s_i = s_i s_j$ when $|i-j| > 1$, and
    \item Braid Relations: $s_i s_{i+1} s_i = s_{i+1} s_i s_{i+1}$. \\
\end{itemize}
In these relations, all indices should be considered mod $N$.

Since the Dynkin diagram is a cycle, it admits a dihedral group's worth of automorphisms.  One can implement a ``flip'' automorphism $\Phi$ by fixing $s_0$ and sending $s_i \rightarrow s_{N-i}$ for all $i \neq 0$, extending the automorphism used in the finite case.  A ``rotation'' automorphism $\rho$ can be implemented by simply sending each generator $s_i \rightarrow s_{i+1}$.  Combinatorially, this corresponds to the following operation.  Given the window notation $(\sigma_1, \sigma_2, \ldots, \sigma_N)$, we have:
\[
\rho(\sigma) = (\sigma_N-N+1, \sigma_1 +1, \sigma_2 +1, \ldots, \sigma_{N-1}+1).
\]
This can be thought of as shifting the base window one place to the left, and then adding one to every entry.  It is clear that this operation preserves the skew periodicity and sum rules for affine permutations, and it is also easy to see that $\rho^N = 1$.

As before, we can define the Hecke algebra of $\ASn_N$, and the $0$-Hecke algebra, generated by $\pi_i$ with $\pi_i$ idempotent anti-sorting operators, exactly mirroring the case for the finite symmetric group.  As in the finite case, elements of the $0$-Hecke algebra are in bijection with affine permutations.  We can also define the $\NDPF$ quotient of $H_0(\ASn_N)$, by introducing the relation 
\[
\pi_{i+1}\pi_i \pi_{i+1} = \pi_{i+1} \pi_i.
\]  
This allows us to give combinatorial definition for the affine $\NDPF$, which we will prove to be equivalent to the quotient.

\begin{definition}
The extended affine non-decreasing parking functions are the functions $f: \ZZ\rightarrow \ZZ$ which are:
\begin{itemize}
    \item Regressive: $f(i)\leq i$,
    \item Order Preserving: $i\leq j \Rightarrow f(i)\leq f(j)$, and
    \item Skew Periodic: $f(i+N) = f(i)+N$.
\end{itemize}
Define the \textbf{shift functions} $\operatorname{sh}_t$ as the functions sending $i \rightarrow i-t$ for every $i$.

The \textbf{affine non-decreasing parking functions} $\ANDPF_N$ are obtained from the extended affine non-decreasing parking functions by removing the shift functions for all $t\neq 0$.  
\end{definition}

Notice that the definition implies that
\[
f(N)-f(1) \leq N.
\]
Furthermore, since the shift functions are not in $\ANDPF_N$, there is always some $j \in \{0, 1, \ldots, N\}$ such that $f(j)\neq f(j+1)$ unless $f$ is the identity.

We now state the main result of this section, which will be proved in pieces throughout the remainder of the chapter.
\begin{theorem}
\label{andpfMainThm}
The affine non-decreasing parking functions $\ANDPF_N$ are a $\JJ$-trivial monoid which can be obtained as a quotient of the $0$-Hecke monoid of the affine symmetric group by the relations $\pi_j\pi_{j+1}\pi_j = \pi_j \pi_{j+1}$, where the subscripts are interpreted modulo $N$.  Each fiber of this quotient contains a unique $[321]$-avoiding affine permutation.
\end{theorem}

\begin{proposition}
\label{andpfgens}
As a monoid, $\ANDPF_N$ is generated by the functions $f_i$ defined by:
\begin{displaymath}
   f_i(j) = \left\{
     \begin{array}{lr}
       j-1 & : j\equiv i+1 (\text{mod }N)\\
       j   & : j\not \equiv i+1 (\text{mod }N).\\
     \end{array}
   \right.
\end{displaymath} 
These functions satisfy the relations:
\begin{align*}
f_i^2 &=& f_i \\
f_if_j &=& f_jf_i \text{ when $|i-j|>1$, and} \\
f_if_{i+1}f_i = f_{i+1}f_if_{i+1} &=& f_{i+1}f_i \text{ when $|i-j|=1$,}
\end{align*}
where the indices are understood to be taken $(\text{mod }N)$.
\end{proposition}
\begin{proof}
One can easily check that these functions $f_i$ satisfy the given relations.  We then check that any $f\in \ANDPF_N$ maybe written as a composition of the $f_i$.

Let $f \in \ANDPF_N$.  If there is no $j \in \{0, \ldots, N\}$ such that $f(j)=f(j+1)$, then $f$ is a shift function, and is thus the identity.

Otherwise, we have some $j$ such that $f(j)=f(j+1)$.  We can then build $f$ using $f_i$'s by the following procedure.  Notice that, if any $g \in \ANDPF$ has $g(j)=g(j+1)$ for some $j$, we can emulate a shift function by concatenating $g$ with $f_j f_{j+1} \cdots f_{j+N-1}$, where the subscripts are understood to be taken $(\text{mod }N)$.  In other words, we have:
\[
g \operatorname{sh}_1 = g f_j f_{j+1} \cdots f_{j+N-1}.
\]

Suppose, without loss of generality, that $f(N)\neq f(N+1)$, so that $N$ and $N+1$ are in different fibers of $f$, and $N$ is maximal in its fiber.  (If the ``break'' occurs elsewhere, we simply use that break as the `top' element for the purposes of our algorithm.  Alternately, we can apply the Dynkin automorphism to $f$ until $\rho^k f(N)\neq \rho^k f(N+1)$. for some $k$.  We can use this algorithm to construct $\rho^k f$, and then apply $\rho$ $N-k$ times to obtain $f$.)  Begin with $g = 1$, and construct $g$ algorithmically as follows.

\begin{itemize}
\item Collect together the fibers.  Set $g'$ to be the shortest element in $\NDPF_N$ such that the fibers of $g'$ match the fibers of $f$ in the base window.  Let $g_0$ be the affine function obtained from a reduced word for $g'$.  This is the pointwise maximal function in $\ANDPF_N$ with fibers equal to the fibers of $f$.

\item Now that the fibers are collected, post-compose $g_0$ with $f_i$'s to move the images into place.  We begin with $g:=g_0$ and apply the following loop:\\
\begin{align*}
&&\text{while } g\neq f: \\
&&\phantom{aaaa}\text{for } i \text{ in } \{1, \ldots, N \}: \\
&&\phantom{aaaaaaaa}\text{if } g(i+1)>f(i+1) \text{ and } g^{-1}(g(i+1)-1) = \emptyset: \\
&&\phantom{aaaaaaaaaaaa}g := g.f_i.
\end{align*}

This process clearly preserves the fibers of $g_0$ (which coincide with the fibers of $f$),  and terminates only if $g=f$.  We need to show that the algorithm eventually halts.

Recall that $g_0(i)\geq f(i)$ for all $i$, and then notice that it is impossible to obtain any $g$ in the evaluation of the algorithm with $g(i)<f(i)$, so that we always have $g(i)-f(i)>0$.  With each application of a $f_j$, the sum $\sum_{i=1}^N (g(i)-f(i))$ decreases by one.

Suppose the loop becomes stuck; then for every $i$ either $f(i+1)=g(i+1)$ or $g^{-1}(g(i+1)-1) \neq \emptyset$.  If there is no $i$ with $f(i+1)=g(i+1)$, then there must be some $i$ with $g^{-1}(g(i+1)-1) = \emptyset$, since $g(N)-g(1)\leq N$ and $g \neq 1$.  Then we can find a minimal $i \in \{1, \ldots, N\}$ with $f(i+1)=g(i+1)$.  

Now, find $j$ minimal such that $f(i+j) \neq g(i+j)$, so that $f(i+j-1) = g(i+j-1)$.  In particular, notice that $i+j-1$ and $i+j$ must be in different fibers for both $f$ and $g$.  If $g^{-1}(g(i+j)-1) = \emptyset$, then the loop would apply a $f_{i+j-1}$ to $g$, but the loop is stuck, so this does not occur and we have that $f(i+j-1)=g(i+j-1)=g(i+j)-1< f(i+j)\leq g(i+j) = g(i+j-1)+1$.  This then forces $g(i+j)=f(i+j)$, contradicting the condition on $j$.

Thus, the loop must eventually terminate, with $g=f$.
\end{itemize}

We have not yet shown that these relations are all of the relations in the monoid; this must wait until we have developed more of the combinatorics of $\ANDPF_N$.  In fact, $\ANDPF_N$ is a quotient of the $0$-Hecke monoid of $\ASn_N$ by the relations $\pi_i \pi_{i+1} \pi_i = \pi_i\pi_{i+1}$ for each $i \in I$, where subscripts are understood to be taken $\text{mod } N$.  To prove this (and simultaneously prove that we have in fact written all the relations in $\ANDPF_N$), we will define three maps, $P, Q$, and $R$ (illustrated in Figure~\ref{fig.affineNDPFMaps}).  The map $P: H_0(\ASn_N)\rightarrow \ANDPF_N$ is the algebraic quotient on generators sending $\pi_i \rightarrow f_i$.  The map $Q: H_0(\ASn_N)\rightarrow \ANDPF_N$ is a combinatorial algorithm that assigns an element of $\ANDPF_N$ to any affine permutation.  In Lemma~\ref{lem.combQuotient} we show that $P=Q$.  Additionally, we have already shown that $P$ is onto (since the $f_i$ generate $\ANDPF_N$), so $Q$ is onto as well.

The third map $R: \ANDPF_N \rightarrow H_0(\ASn_N)$ assigns a $[321]$-avoiding affine permutation to an $f\in \ANDPF_N$.  In fact, $R \circ P$ is the identity on the set of $[321]$-avoiding affine permutations, and $P\circ R$ is the identity on $\ANDPF_N$.  This then implies that there are no additional relations in $\ANDPF_N$.
\end{proof}

\begin{figure}
  \begin{center}
  \includegraphics[scale=.75]{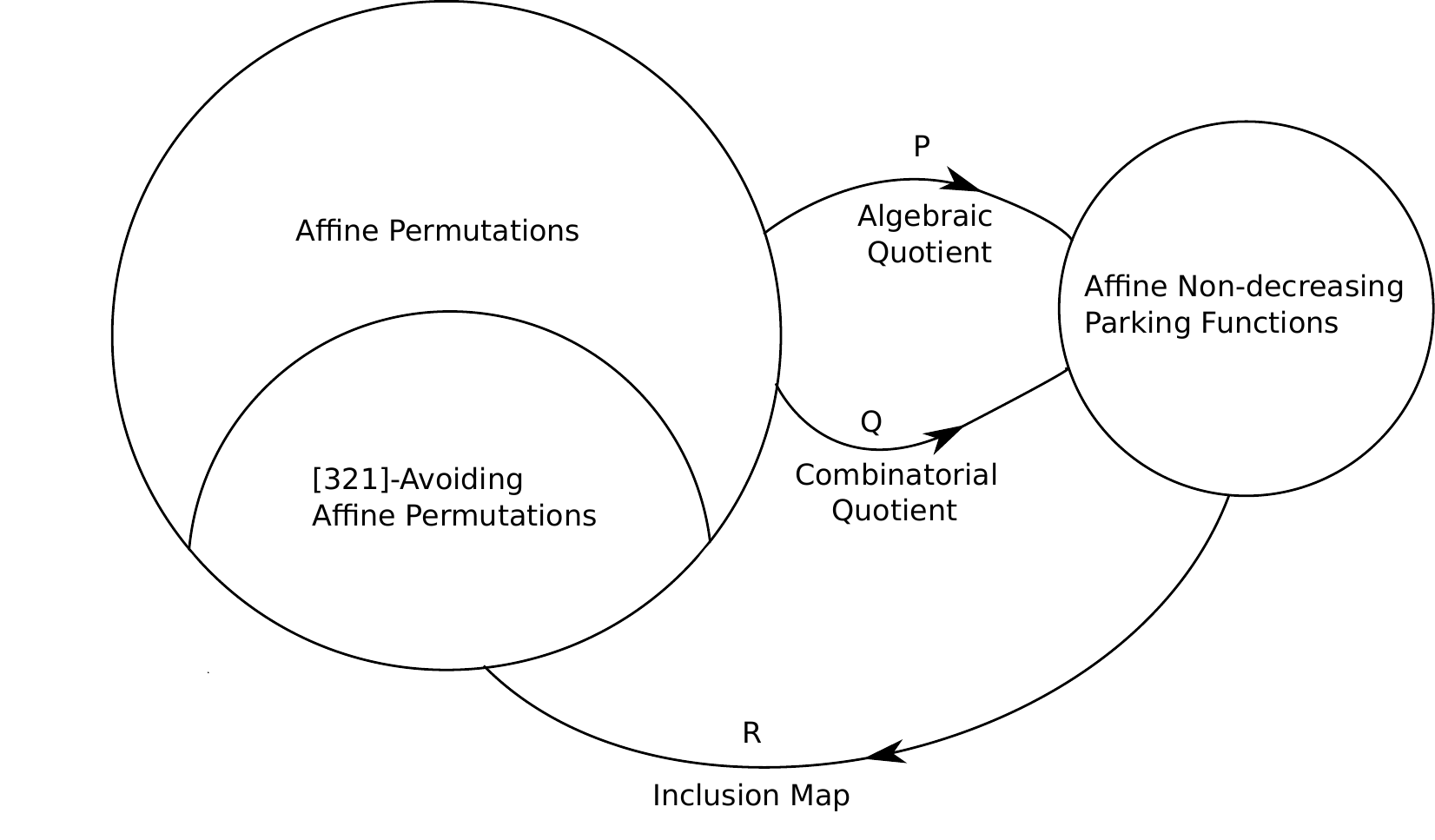}
  \end{center}
  \caption{Maps between $H_0(\ASn_N)$ and $\ANDPF_N$.}
  \label{fig.affineNDPFMaps}
\end{figure}

\begin{corollary}
The map $P: H_0(\ASn_N)\rightarrow \ANDPF_N$, defined by sending $\pi_i \rightarrow f_i$ and extending multiplicatively, is a monoid morphism.
\end{corollary}
\begin{proof}
The generators $f_i$ satisfy all relations in the $0$-Hecke algebra, so $P$ is a quotient of $H_0(\ASn_N)$ by whatever additional relations exist in $\ANDPF_N$.
\end{proof}

\begin{lemma}
Any function $f \in \ANDPF_N$ is entirely determined by its set of fibers, set of images, and one valuation $f(i)$ for some $i\in \ZZ$.
\end{lemma}
\begin{proof}
This follows immediately from the fact that $f$ is regressive and order preserving.
\end{proof}

\begin{lemma}
Let $f \in \ANDPF_N$, and $F_f=\{m_j\}$ be the set of maximal elements of the fibers of $f$.  Each pair of distinct elements $m_j, m_k$ of the set $F_f \cap \{1, 2, \ldots, N\}$ has $f(m_j) \not \equiv f(m_k) (\text{mod } N)$.
\end{lemma}
\begin{proof}
Suppose not.  Then $f(m_j)-f(m_i) = kN$ for some $k \in \ZZ$, implying that 
$f(m_j)=f(m_i+kN)$.  Since $f(m_j)-f(m_i)\leq N$, we must have $k=0$.  But then $m_j$ and $m_i$ are in the same fiber, providing a contradiction.
\end{proof}

\begin{theorem}
$\ANDPF_N$ is $\JJ$-trivial.
\end{theorem}
\begin{proof}
Thi is a direct consequence of the regressiveness of functions in $\ANDPF_N$.  Let $M:= \ANDPF_N$, and $f\in M$.  Then each $g\in MfM$ has $g(i)\leq f(i)$ for all $i \in \ZZ$.  Thus, if $MgM=MfM$, we must have $f=g$.  Then the $\JJ$-equivalence classes of $M$ are trivial, so $\ANDPF_N$ is $\JJ$-trivial.
\end{proof}

Note that $\ANDPF_N$ is not aperiodic in the sense of a finite monoid.  (Aperiodicity was defined in Section~\ref{sec:bgnot}.)  Take the function $f$ where $f(i)=0$ for all $i \in \{1, \ldots, N\}$.  Then $f^k(1) = (1-k)N$, so there is no $k$ such that $f^k = f^{k+1}$.

\subsection{Combinatorial Quotient}
\label{subsec:combintorialQuotient}

A direct combinatorial map from affine permutations to $\ANDPF_N$ is now discussed.  This map directly constructs a function $f$ from an arbitrary affine permutation $x$, with the same effect as applying the algebraic $\ANDPF$ quotient to the $0$-Hecke monoid element indexed by $x$.  We first define the combinatorial quotient in the finite case and provide an example (Figure~\ref{fig.combQuotient}).

\begin{definition}
The combinatorial quotient $Q_{cl}: H_0(S_N) \rightarrow \NDPF_N$ is given by the following algorithm, which assigns a function $f$ to a permutation $x$.
\begin{enumerate}
\item Set $f(N):=x(N)$.
\item Suppose $i$ is maximal such that $f(i)$ is not yet defined.  If $x(i)<x(i+1)$, set $f(i):=f(i+1)$.  Otherwise, set $f(i):=x(i)$.
\end{enumerate}
\end{definition}

\begin{figure}
  \begin{center}
  \includegraphics[scale=1]{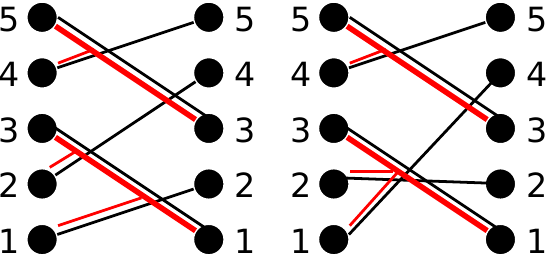}
  \end{center}
  \caption{Example of the combinatorial quotient $Q_{cl}: H_0(S_5) \rightarrow \NDPF_5$.  The string diagram is read left-to-right, with the permutation illustrated with black strings and the image function drawn in red.  The permutation in the left diagram, then, is $x=[2,4,1,5,3]$ and $Q_{cl}(x)$ is the function $f=[1,1,1,3,3]$.  For the permutation on the right, we have $y=[4,2,1,5,3]$ and $Q_{cl}(y)=Q_{cl}(x)=[1,1,1,3,3]$.  Notice that these two permutations $x$ and $y$ are related by turning the $[321]$-pattern in $y$ into a $[231]$-pattern in $x$, preserving the fiber of $Q$.}
  \label{fig.combQuotient}
\end{figure}

Note that the map $Q_{cl}$ is closely related to bijection of Simion and Schmidt between $[132]$-avoiding permutations and $[123]$-avoiding permutations~\cite{simion.schmidt.1985}.  (The bijection is also covered very nicely in~\cite{bona.permutations})  This bijection operates by marking all left-to-right minima (ie, elements smaller than all elements to their left) of a $[132]$-avoiding permutation, and then reverse-sorting all elements which are not marked. The resulting permutation is $[123]$-avoiding.  For example, the permutation
$[\textbf{5}, 6, \textbf{4}, 7, \textbf{1}, 2, 3]$ avoids the pattern $[132]$; the bold entries are the left-to-right minima.  Sorting the non-bold entries, one obtains the permutation
$[\textbf{5}, 7, \textbf{4}, 6, \textbf{1}, 3, 2]$, which avoids the permutation $[123]$.  Notice that the bold entries are still left-to-right minima after anti-sorting the other entries.

The patterns $[231]$ and $[123]$ are the respective ``reverses'' of the patterns $[132]$ and $[321]$, obtained by simply reversing the one-line notation.  It is trivial to observe that $x$ avoids $p$ if and only if the reverse of $x$ avoids the reverse of $p$.  Then the ``reverse'' of the Simion-Schmidt algorithm (which marks right-to-left minima, and sorts the other entries) gives a bijection between $[231]$- and $[321]$-avoiding permutations; in fact, this is the same bijection given by the fibers of the $\NDPF$ quotient of the $0$-Hecke monoid.

A similar combinatorial quotient may be defined from $\ASn_N\rightarrow \ANDPF_N$, generalizing the map $Q_{cl}$.  This map will assign a function $f$ to an affine permutation $x$.

Below, we will show that each fiber of the map $Q$ contains a unique $[321]$-avoiding affine permutation (Theorem~\ref{thm:affNdpfFibers321}).  However, it is too much to expect a bijection between affine $[231]$- and $[321]$-avoiding permutations.  By a result of Crites, there are infinitely many affine permutations that avoid a pattern $\sigma$ if and only if $\sigma$ contains the pattern $[321]$~\cite{Crites.2010}.  Thus, there are infinitely many $[321]$-avoiding affine permutations, but only finitely many $[231]$-avoiding affine permutations.

We first identify some $k \in \{1, 2, \ldots, N\}$ such that for every $j>k$, $x(j)>x(k)$.  

\begin{lemma}
Let $k_0 \in \{1, 2, \ldots, N\}$ have $x(k_0)\leq x(m)$ for every $m \in \{1, 2, \ldots, N\}$.  Then for every $j>k_0$, $x(j)>x(k_0)$.
\end{lemma}
\begin{proof}
 Suppose $j>k_0$ with $x(j)<x(k_0)$.  Then there exists $p\in \N$ such that $j-pN \in \{1, 2, \ldots, N\}$, so that $x(j-pN) = x(j)-pN < x(k_0)$, contradicting the minimality of $x(k_0)$.
\end{proof}

Now the affine combinatorial quotient is defined by the following algorithm.

\begin{definition}
The combinatorial quotient $Q: H_0(\ASn_N) \rightarrow \ANDPF_N$ is given by the following algorithm, which assigns a function $f$ to an affine permutation $x$.
\begin{enumerate}
\item Let $k_0 \in \{1, 2, \ldots, N\}$ have $x(k_0)\leq x(m)$ for every $m \in \{1, 2, \ldots, N\}$.  Set $f(k_0)=x(k_0)$.
\item Choose $i \in \{1, 2, \ldots, N-1\}$ minimal such that $f(k_0 - i)$ is not yet defined.  If $x(k_0-i+1)<x(k_0-i)$, set $f(k_0-i) := f(k_0-i+1)$.  Otherwise, set $f(k_0-i) := x(k_0-i)$.
\item Define $f$ on all other $i$ using skew periodicity.
\end{enumerate}
\end{definition}

\begin{lemma}
\label{lem.combQuotient}
The affine combinatorial quotient $Q$ agrees with the algebraic $\ANDPF$ quotient $P$. 
\end{lemma}
\begin{proof}
We denote the combinatorial quotient by $Q$ and the algebraic quotient by $P$.  

One can easily check that $Q(1)=P(1)=1$, and $Q(\pi_i)=P(\pi_i)=f_i$.  Since $P$ is a monoid morphism, we have that $P(x\pi_i)=P(x)P(\pi_i)= ff_i$.   We then assume that $Q(x)=P(x)=f$, and consider $Q(x\pi_i)$.  We will show that $Q(x\pi_i)= Q(x)f_i = ff_i = P(x\pi_i)$.  

If $\pi_i$ is a right descent of $x$ then $Q(x\pi_i)=Q(x)=f=P(x\pi_i)$, and we are done.

If $\pi_i$ is not a right descent of $x$, we have $x(kN+i) < x(kN+i+1)$ for all $k\in \ZZ$, and 
\[
x\pi_i(j) = 
    \left\{
     \begin{array}{lr}
       x(j)   \text{ for all } j\not \equiv i, i+1 (\text{mod } N)\\
       x(j+1) \text{ for all } j     \equiv i      (\text{mod } N)\\
       x(j-1) \text{ for all } j     \equiv    i+1 (\text{mod } N)\\
     \end{array}
   \right.
\]  
We examine the functions $Q(x\pi_i)$ and $ff_i$ on $i$ and $i+1$, since these functions are equal on $j\not \equiv i, i+1 (\text{mod } N)$, and the actions on $i$ and $i+1$ then determine the functions on all $j \equiv i, i+1 (\text{mod } N)$. 

We consider two cases, depending on whether $i$ and $i+1$ are in the same fiber of $f$.  

\begin{itemize}
\item If $i$ and $i+1$ are in the same fiber of $f$ and $i+1$ is maximal in this fiber, we must (by construction of $Q$) have $x(i+1)<x(i)$, contradicting the assumption that $\pi_i$ was not a right descent of $x$.

\item If $i$ and $i+1$ are in the same fiber of $f$ and $i+1$ is not maximal in this fiber, then there exists some (minimal) $m> i+1>i$ with $x(m)<x(i)$ and $x(m)<x(i+1)$, maximal in the fiber of $i$ and $i+1$.  Then $x(m)<x(i+1)=x\pi_i(i)$ and $x(m)<x(i)=x\pi_i(i+1)$.  Since the maximal size of a fiber of $f$ is $N$, we have that $m-i\leq N$.  Then (since $i+1$ not maximal in the fiber of $f$) $m\not \equiv i+1 (\text{mod } N)$.

If $m \equiv i (\text{mod } N)$, then $i$ is maximal in its fiber, and we must have $i$ and $i+1$ in different fibers, contrary to assumption.

If $m\not \equiv i (\text{mod } N)$, we have $x(m)=x\pi_i(m) < x\pi_i(i), x\pi_i(i+1)$, and so by the construction of $Q$, we have 
$Q(x\pi_i)(i)=Q(x\pi_i)(i+1)=Q(x\pi_i)(m)=Q(x)(m) = x(m)$.  Then in this case, $Q(x\pi_i)=f$.

On the other hand, $ff_i(i)=f(i)=f(m)=ff_i(m)$, and $ff_i(i+1)=f(i)=f(m)=ff_i(m)$, so $ff_i=f$.

\item If $i$ and $i+1$ are in different fibers of $f$, then we have $i$ maximal in its fiber, and take $m$ (possibly equal to $i+1$) to be the maximal element of the fiber in which $i+1$ sits.  We note that if $m\equiv i+1 (\text{mod } N)$, then we must have $i$ and $i+1$ in the same fiber, reducing to the previous case.

Otherwise, applying the construction of $Q$, we find that $Q(x\pi_i)(i+1) = x(i)$, and that $Q(x\pi_i)(i) = x(i)$; thus $i+1$ is removed from its fiber and merged into the fiber with $i$.  The resulting function is equal to $ff_i$.
\end{itemize}
This exhausts all cases, completing the proof.
\end{proof}

\begin{corollary}
The finite type combinatorial quotient agrees with the $\NDPF_N$ quotient of $H_0(S_N)$ obtained by introducing the relations $\pi_i \pi_{i+1}\pi_i = \pi_{i+1} \pi_i$ for $i\in \{1, \ldots, N-2\}$. 
\end{corollary}
\begin{proof}
This follows immediately from Lemma~\ref{lem.combQuotient} by parabolic restriction to the finite case.  In the finite case, the index set is $\{1, 2, \ldots, N-1\}$, so we must have $i\in \{1, \ldots, N-2\}$.
\end{proof}

\subsection{Affine $[321]$-Avoidance}

An affine permutation $x$ avoids a pattern $\sigma \in S_k$ if there is no subsequence of $x$ in the same relative order as $\sigma$.  This ostensibly means that an infinite check is necessary, however one may show that only a finite number of comparisons is necessary to determine if $x$ contains a $[321]$-pattern.  The following lemma is equivalent to~\cite[Lemma 2.6]{Green.2002}.

\begin{lemma}
Let $x$ contain at least one $[321]$-pattern, with $x_i> x_j> x_k$ and $i<j<k$.  Then $x$ contains a $[321]$-pattern $x_{i'}> x_j> x_{k'}$ such that $i\leq i'<j<k'\leq k$, $j-i'< N$, and $k'-j < N$.
\end{lemma}
\begin{proof}
We have $x_j>x_k>x_{k-aN} = x_k - aN$ for $a \in \N$, so if $k-j>N$, we can find a $[321]$ pattern replacing $x_k$ with $x_{k-aN}$.  A similar argument allows us to replace $i$ with $i+bN$ for the maximal $b\in \N$ such that $j-(i+bN) < N$.
\end{proof}

\begin{figure}
  \begin{center}
  \includegraphics[scale=1]{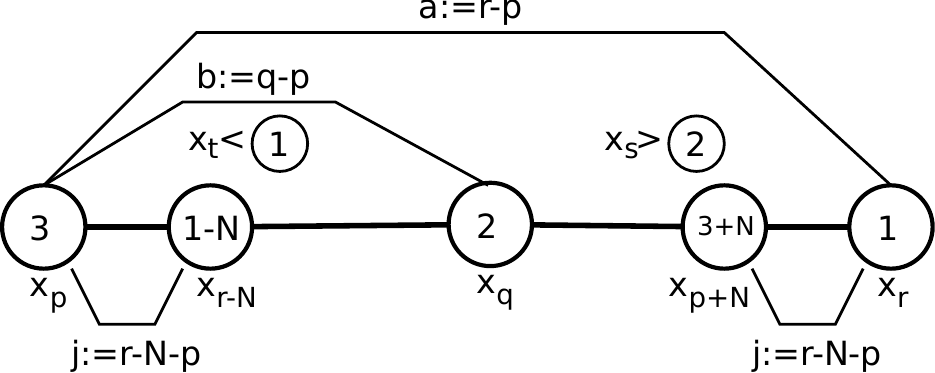}
  \end{center}
  \caption{Diagram of a bountiful width system for the pattern $[321]$ for affine permutations.  The pattern occurs at positions $(x_p, x_q, x_r)$, with width system given by $(r-p, q-p)$.  In the case where $r-p>N$, there is an `overlap' of $j=r-N-p$.  Bountifulness of the width system ensures that the elements in the overlap may be moved moved out of the interior of the pattern instance by a sequence of simple transpositions, each decreasing the length of the permutation by one, just as in the non-affine case.}
  \label{fig.minimalAffine321}
\end{figure}

As noted by Green, one can then check whether an affine permutation contains a $[321]$-pattern using at most $\binom{N}{3}$ comparisons.  Green also showed that any affine permutation containing a $[321]$-pattern contains a braid; we can actually replicate this result using a width system on the affine permutation, as depicted in Figure~\ref{fig.minimalAffine321}.  The Lemma ensures that the width of a minimal $[321]$-pattern under this width system has a total width of at most $2N-2$.  One must consider the case when the total width of a minimal $[321]$-instance is greater than $N$, but nothing untoward occurs in this case: the width system is bountiful and allows a factorization of $x$ over $[321]$.

We now prove the main result of this section.

\begin{theorem}
\label{thm:affNdpfFibers321}
Each fiber of the $\ANDPF_N$ quotient of $\ASn_N$ contains a unique $[321]$-avoiding affine permutation.
\end{theorem}
\begin{proof}
We first establish that each fiber contains a $[321]$-avoiding affine permutation, and then show that this permutation is unique.

Recall the algebraic quotient map $P: H_0(\ASn_N) \rightarrow \ANDPF_N$, which introduces the relation $\pi_i \pi_{i+1}\pi_i = \pi_{i+1} \pi_i$.  

Choose an arbitrary affine permutation $x$; we show that the fiber $Q^{-1}\circ Q(x)$ contains a $[321]$-avoiding permutation.  If $x$ is itself $[321]$-avoiding, we are already done.  So assume $x$ contains a $[321]$-pattern.  As shown by Green~\cite{Green.2002}, an affine permutation $x$ contains a $[321]$-pattern if and only if $x$ has a reduced word containing a braid; thus, $x=y \pi_i \pi_{i+1} \pi_i z$ for some permutations $y$ and $z$ with $\len(x)=\len(y)+3+\len(z)$.  Applying the $\ANDPF_N$ relations, we may set $x'= y \pi_{i+1} \pi_i z$, and have $Q(x)=Q(x')$, with $\len(x')=\len(x)-1$.  If $x'$ contains a $[321]$, we apply this trick again, reducing the length by one.  Since $x$ is of finite length, this process must eventually terminate; the permutation at which the process terminates must then be $[321]$-avoiding.  Then the fiber $Q^{-1}\circ Q(x)$ contains a $[321]$-avoiding permutation.

We now show that each fiber contains a unique $[321]$-avoiding affine permutation, using the combinatorial quotient map.

Let $x$ be $[321]$-avoiding, and let $Q(x)=f$ an affine non-decreasing parking function; we use information from $f$ to reconstruct $x$.  Let $\{m_i\}$ be the set of elements of $\ZZ$ that are maximal in their fibers under $f$.  By the construction of the combinatorial quotient map, we have $x(m_i)=f(m_i)$ for every $i$.  Since $f$ is in $\ANDPF_N$, we have $x(m_i)<x(m_{i'})$ whenever $i<i'$; thus $\{x(m_i)\}$ is a strictly increasing sequence.

Let $\{m_{i,j}\} = f^{-1}\circ f(m_i) \setminus \{m_i\}$, with $m_{i,j}<m_{i,j+1}$ for every $j$.  Notice that if $i<i'$ and $j<j'$ then $m_{i,j}<m_{i',j'}$.  

We claim that if $i<i'$ and $j<j'$, then $x(m_{i,j})<x(m_{i',j'})$.  If not, then we have 
\[
x(m_{i'}) < x(m_{i',j'}) < x(m_{i,j}), \text{ with }  m_{i,j} < m_{i',j'} < m_{i'},
\]
in which case $x$ contains a $[321]$-pattern, contrary to assumption.  Thus, the sequence $\{x(m_{i,j})\}$ with $i$ and $j$ arbitrary is a strictly increasing sequence.

Now $\{f(m_i) = x(m_{i})\}$ and $\{x(m_{i,j})\}$ are two increasing sequences.  Since $x$ is a bijection, and every $z\in \ZZ$ is either an $m_i$ or an $m_{i,j}$, $x$ is determined by the choice of $x(m_{1,1})$.  A valid choice for $x(m_{1,1})$ exists, since every $f$ arises as the image of some affine permutation under $Q$, and every fiber contains some $[321]$-avoiding element.  

One can show that the choice of $x(m_{1,1})$ is uniquely determined by the following argument.  Suppose two valid possibilities exist for $x(m_{1,1})$, giving rise to two different $[321]$-avoiding affine permutations $x$ and $x'$.  Suppose without loss of generality that $1\leq m_{1,1}\leq N$, and that $x(m_{1,1})<x'(m_{1,1})$.  Then: 
\begin{align*}
\binom{N+2}{2} &=&  \sum_{k=1}^{N} x(k) \\
 &=& \sum ( x(m_i) + \sum x(m_{i,j}) ) \text{ where $m_i$, $m_{i,j} \in \{1, \ldots, N\}$ }\\
 &<& \sum ( x'(m_i) + \sum x'(m_{i,j}) ) \text{ where $m_i$, $m_{i,j} \in \{1, \ldots, N\}$ }\\
 &=& \sum_{k=1}^{N} x'(k) \\
 &=& \binom{N+2}{2},
\end{align*}
providing a contradiction.  Hence $x(m_{1,1})$ is uniquely determined, and thus each fiber of $Q$ contains a unique $[321]$-avoiding permutation.
\end{proof}

\bibliographystyle{amsalpha}

\bibliography{DissertationBibliography}

\end{document}